\documentclass[11pt,reqno]{amsart}
\usepackage[a4paper,margin=1.05in]{geometry}
\usepackage{amsmath,amssymb,amsthm,mathrsfs}
\usepackage{xcolor}
\usepackage{hyperref}

\hypersetup{
    colorlinks=true,
    linkcolor=blue,
    citecolor=blue,
    urlcolor=blue
}
\theoremstyle{plain}
\newtheorem{theorem}{Theorem}[section]
\newtheorem{proposition}[theorem]{Proposition}
\newtheorem{lemma}[theorem]{Lemma}
\newtheorem{corollary}[theorem]{Corollary}
\theoremstyle{definition}
\newtheorem{definition}[theorem]{Definition}
\newtheorem{remark}[theorem]{Remark}
\newtheorem{example}[theorem]{Example}
\usepackage{fancyhdr}
\newcommand{\R}{\mathbb{R}}
\newcommand{\N}{\mathbb{N}}
\newcommand{\dHs}{\dot H^{s}(\R^N)}
\newcommand{\Lam}{\Lambda_{N,s}}
\newcommand{\Psins}{\Psi_{N,s}}
\newcommand{\mubar}{\bar\mu_{N,s}}
\newcommand{\cstar}{2^{*}_{s}}

\begin{document}

\title[Fractional Brezis-Nirenberg with multipolar Hardy potentials]
{A fractional Brezis-Nirenberg problem with multipolar Hardy potentials}
\author{Debangana Mukherjee}
\address{Department of Mathematics, School of Interwoven Arts and Sciences (SIAS),
Krea University, Sri City, Andhra Pradesh, India}
\email{debangana.mukherjee@krea.edu.in}
\maketitle

\begin{abstract}
We study a fractional Brezis--Nirenberg problem on a bounded domain with finitely many Hardy potentials centered at distinct points $a_1,\dots,a_k$, with masses $\mu_1,\dots,\mu_k\geq 0$. We assume
$\sum_{i=1}^k\mu_i<\Lambda_{N,s}$
and consider $N>4s$ and $0<\mu_{\max}<\overline{\mu}_{N,s}$. The interaction between the poles allows existence for $\lambda=0$ and for some negative values of $\lambda$. At a pole $a_j$, this interaction is described by
$B_j(\lambda,\mathbf a,\boldsymbol\mu)=\lambda+\sum_{i\neq j}\frac{\mu_i}{|a_i-a_j|^{2s}}.$
This leads to a threshold $\lambda_{\mathrm{geo}}\leq 0$. We prove that the critical Rayleigh quotient is attained for
$\lambda_{\mathrm{geo}}<\lambda<\lambda_1(\boldsymbol\mu,\mathbf a),$
and hence the problem admits a positive least-energy solution. If at least two masses are positive, then $\lambda_{\mathrm{geo}}<0$, so the result includes $\lambda=0$ and a nonempty interval of negative values of $\lambda$. We also prove a direct compactness estimate for weakly vanishing sequences by localizing the Gagliardo seminorm. The argument avoids both profile decomposition and the Caffarelli-Silvestre extension.
\end{abstract}
%\tableofcontents
\section{Introduction and main results}\label{sec:intro}

\subsection{The fractional multipolar Brezis--Nirenberg problem}\label{sec:problem}
On a bounded domain, a semilinear equation with critical Sobolev exponent need not
have a positive solution, and Brezis and Nirenberg showed that adding a linear term
$\lambda u$ restores existence for $\lambda$ in a suitable range. We study a
fractional version of this question in which the equation also involves several
inverse-power Hardy potentials, each singular at an interior point of the domain. Let $N\in\N$, $s\in(0,1)$ with $N>2s$, and let $\Omega\subset\R^N$ be a bounded connected open set with smooth boundary. Fix $k\ge 1$ distinct points
$a_1,\dots,a_k\in\Omega,\, a_i\neq a_j\ (i\neq j),\, d_j:=\min_{i\neq j}|a_i-a_j|\ \ (j=1,\dots,k),\, d:=\min_{j}d_j>0$. With the convention $d_1:=+\infty$ when $k=1$, so any upper-bound condition involving $d_1$ or $d$ is immediate in that case. The masses are $\mu_1,\dots,\mu_k\ge 0$. Write $\mathbf a=(a_1,\dots,a_k)$,
$\boldsymbol\mu=(\mu_1,\dots,\mu_k)$, $\mu_{\max}=\max_i\mu_i$, and $\cstar=\frac{2N}{N-2s}$. 
For $\lambda\in\R$ consider
\begin{equation}\label{eq:P}
\left\{
\begin{aligned}
(-\Delta)^s u-\sum_{i=1}^{k}\mu_i\,\frac{u}{|x-a_i|^{2s}}
&=\lambda u+|u|^{\cstar-2}u, && x\in\Omega,\\
u&=0, && x\in\R^N\setminus\Omega .
\end{aligned}
\right.
\tag{$P_{\lambda,\mathbf a}$}
\end{equation}
The associated quadratic form and energy functional are
\begin{equation}\label{eq:Q}
Q_{\boldsymbol\mu,\mathbf a}(u):=[u]_s^2-\sum_{i=1}^{k}\mu_i
\int_{\Omega}\frac{u^2}{|x-a_i|^{2s}}\,dx ,
\end{equation}
\begin{equation}\label{eq:J}
J_\lambda(u):=\tfrac12 Q_{\boldsymbol\mu,\mathbf a}(u)
-\tfrac{\lambda}{2}\int_\Omega u^2\,dx-\tfrac{1}{\cstar}\int_\Omega|u|^{\cstar}dx .
\end{equation}
Throughout we assume
\begin{equation}\label{eq:H}
\sum_{i=1}^{k}\mu_i<\Lam ,
\tag{H}
\end{equation}
where $\Lam$ is the sharp constant in the fractional Hardy inequality, recalled in
Theorem~\ref{thm:hardy} below. Problem \eqref{eq:P} is related to four relevant directions of previous work which we briefly
recall.

\emph{Brezis-Nirenberg.} For $s=1$ and $k=0$, Brezis and Nirenberg \cite{BN}
studied $-\Delta u=\lambda u+|u|^{2^*-2}u$ on a bounded domain and showed that a
positive solution exists for every $\lambda\in(0,\lambda_1)$ when $N\ge4$, while for
$N=3$ the admissible range of $\lambda$ is strictly smaller. The role of the dimension is determined by the integrability of the Aubin–Talenti profile and the dimension dependence is reflected below in the condition $\beta_+(\mu)>\frac N2$.

\emph{Fractional Brezis-Nirenberg.} The nonlocal counterpart with $k=0$ was obtained by Servadei and Valdinoci for $N\ge4s$ in \cite{SV-BN} and in the
low-dimensional case $2s<N<4s$ in \cite{SV-low}. Further existence and multiplicity
results including sign-changing solutions for fractional Brezis-Nirenberg problems
with an additional subcritical term are obtained in \cite{Muk}.

\emph{One pole.} For $s=1$ and $k=1$, Jannelli \cite{Ja} identified the role of
the Hardy parameter, extremals exist for every $0<\lambda<\lambda_1(\mu)$ whenever
$0\le\mu\le\frac{(N-2)^2}{4}-1$. The remaining range requires a different
argument followed from Ghoussoub and Robert \cite{GR16} through a notion of mass. Precise local
asymptotics at the singularity and unique continuation for fractional elliptic
equations with Hardy-type homogeneous potentials were established by Fall and Felli
\cite{FF}. In the fractional setting, attainability of the best constant in the
critical Hardy-Sobolev inequality and the doubly critical problem on $\R^N$ are
treated by Ghoussoub and Shakerian \cite{GS}. The counterpart of Jannelli's result is due to
Ghoussoub, Robert, Shakerian and Zhao \cite{GRSZ}, below a threshold $\gamma_{\mathrm{crit}}$ on the mass, least-energy solutions exist for every $\lambda\in(0,\lambda_1)$. But above this threshold, existence depends on the positivity of a fractional Hardy-Schr\"odinger mass associated with $\Omega$. Bounded-domain fractional problems with a Hardy potential were considered earlier by Barrios, Medina and Peral \cite{BMP} and fractional Brezis-Nirenberg systems with Hardy potentials in \cite{Shen}.

\emph{Several poles.} For $s=1$, Felli and Terracini \cite{FT} treat both the whole-space and the bounded-domain multipolar problems and give conditions on the masses and on the locations of the singularities for the associated Rayleigh quotient to be attained. Cao and Han \cite{CH} consider the corresponding critical problem
on bounded domains. In the fractional setting, positivity of multipolar
Schr\"odinger operators and its dependence on the configuration of the poles are
studied by Felli, Mukherjee and Ognibene \cite{FMO}, and a concentration-compactness
argument for multipolar Hardy potentials, in the extension formulation, is developed
in \cite{MMO}.

\emph{Main contributions}. The main contributions of the paper are twofold, and both arise from the interaction between the poles rather than from a direct extension of the single-pole problem.

The first is that the poles interact, and that the interaction has a sign. In every existence result quoted above the linear term must be strictly positive, $\lambda>0$. Here the Hardy potentials centred at the poles $a_i$, $i\neq j$, act near $a_j$ as an additional positive weight, and only the combination
\[
B_j(\lambda,\mathbf a,\boldsymbol\mu)
=\lambda+\sum_{i\neq j}\frac{\mu_i}{|a_i-a_j|^{2s}}
\]
has to be positive. This produces a threshold $\lambda_{\mathrm{geo}}\le0$, determined
by the geometry of the configuration alone, and the existence range is
$\lambda>\lambda_{\mathrm{geo}}$ rather than $\lambda>0$. As soon as two masses are
positive the inequality $\lambda_{\mathrm{geo}}<0$ is strict, so \eqref{eq:P} has a
positive least-energy solution for $\lambda=0$ and for an interval of negative
$\lambda$, whose lower endpoint $\lambda_{\mathrm{geo}}$ moves arbitrarily far to the
left as the poles approach one another. For $k=1$
one has $B_j=\lambda$ and $\lambda_{\mathrm{geo}}=0$, and the phenomenon disappears.

The second is the form of the compactness threshold and the way it is reached. We
prove that the compactness constant for \eqref{eq:P} is
$S_*=\min\{S,S_{\mu_1},\dots,S_{\mu_k}\}=S_{\mu_{\max}}$, by a localization of the
Gagliardo seminorm against a partition of unity adapted to the poles, in which the
error is controlled by the $L^2$ norm alone and therefore vanishes along a weakly
vanishing sequence. The argument uses neither a profile decomposition, nor a
classification of limit profiles, nor the Caffarelli-Silvestre extension, and each
pole contributes its own constant $S_{\mu_i}$ through a single application of the
one-pole inequality.

\subsection{Main results and the geometric threshold}\label{sec:mainresults}
We use the notation of Section~\ref{sec:notation}. $X$ denotes the natural variational space associated with the problem,
$[\,\cdot\,]_s$ the Gagliardo seminorm, and $S_\mu$ the one-pole Hardy-Sobolev
constant \eqref{eq:Smu}, with $S=S_0$. Let $\lambda_1(\boldsymbol\mu,\mathbf a)$
denote the principal eigenvalue of the multipolar quadratic form, defined in
Subsection~\ref{sec:eigen}, and let $\mubar$ be the explicit threshold of
Theorem~\ref{thm:mubar}, which satisfies $0<\mubar<\Lam$. Define
\begin{equation}\label{eq:lambdageo}
J_*:=\{j:\mu_j=\mu_{\max}\}, \,
\lambda_{\mathrm{geo}}
:=-\max_{j\in J_*}\ \sum_{i\neq j}\frac{\mu_i}{|a_i-a_j|^{2s}}\ \le\ 0 ,
\end{equation}
and, for $\lambda<\lambda_1(\boldsymbol\mu,\mathbf a)$,
\begin{equation}\label{eq:Slambda}
S_{\lambda,\boldsymbol\mu,\mathbf a}(\Omega)
:=\inf_{u\in X\setminus\{0\}}
\frac{Q_{\boldsymbol\mu,\mathbf a}(u)-\lambda\|u\|^2_{L^2(\Omega)}}
{\|u\|^2_{L^{\cstar}(\Omega)}},
\end{equation}
\begin{equation}\label{eq:Sstar}
S_*:=\min\{S,S_{\mu_1},\dots,S_{\mu_k}\},\,
c_*:=\frac sN\,S_*^{N/(2s)} .
\end{equation}

\begin{theorem}\label{thm:main}
Assume \eqref{eq:H}, $N>4s$ and $0<\mu_{\max}<\mubar$. Then for every
\[
\lambda_{\mathrm{geo}}<\lambda<\lambda_1(\boldsymbol\mu,\mathbf a)
\]
the quotient $S_{\lambda,\boldsymbol\mu,\mathbf a}(\Omega)$ is attained, and
\eqref{eq:P} admits a least-energy weak solution $u$ with $u>0$ a.e.\ in $\Omega$ and
$J_\lambda(u)<c_*$. 
If two or more coefficients $\mu_i$ are positive, then $\lambda_{\mathrm{geo}}<0$, so the existence range includes $\lambda=0$ and a nonempty interval of negative values of $\lambda$.
\end{theorem}

\begin{example}[two poles]\label{ex:twopoles}
Let $k=2$ and $N>4s$, assume
\[
0<\mu_2<\mu_1<\mubar,\, \mu_1+\mu_2<\Lam ,
\]
and set $d:=|a_1-a_2|$. Then $\mu_1=\mu_{\max}$, so $J_*=\{1\}$ and
\[
\lambda_{\mathrm{geo}}=-\frac{\mu_2}{d^{2s}} .
\]
Theorem~\ref{thm:main} gives a least-energy solution, positive a.e.\ in $\Omega$, for
every $-\mu_2d^{-2s}<\lambda<\lambda_1(\boldsymbol\mu,\mathbf a)$. This range contains
$\lambda=0$ and the nonempty interval $\big(-\mu_2d^{-2s},0\big)$ so the second pole
admits negative values of $\lambda$. Moreover, the lower threshold $-\mu_2d^{-2s}$
decreases as the distance $d$ between the poles decreases, so the admissible interval
below $0$ becomes larger as the two poles move closer together. We make no claim about
the full interval, whose upper endpoint $\lambda_1(\boldsymbol\mu,\mathbf a)$ also
depends on the configuration of the poles.
\end{example}

The hypothesis $\mu_{\max}>0$ excludes only $\boldsymbol\mu=0$, in which case
\eqref{eq:P} is the fractional Brezis-Nirenberg problem without Hardy terms treated
for $N\ge4s$ in \cite{SV-BN} and for $2s<N<4s$ in \cite{SV-low}. There
$\lambda_{\mathrm{geo}}=0$ and the range in question is $0<\lambda<\lambda_1$. The
strict inequality $S_{\lambda,\boldsymbol\mu,\mathbf a}(\Omega)<S_*$ used in Theorem~\ref{thm:main} follows from an expansion of the Rayleigh quotient
along a
concentrating family truncated at a dominant pole $a_j$. The effect of the linear term and the other poles is given by the coefficient
\begin{equation}\label{eq:Bjintro}
B_j(\lambda,\mathbf a,\boldsymbol\mu)
=\lambda+\sum_{i\neq j}\frac{\mu_i}{|a_i-a_j|^{2s}} ,
\end{equation}
and $B_j>0$ for some $j\in J_*$ exactly when $\lambda>\lambda_{\mathrm{geo}}$
(Theorem~\ref{thm:expansion} and Corollary~\ref{cor:strict}).

To deduce attainment from this strict inequality, we establish the following
compactness result, the main technical result of the paper.

\begin{theorem}\label{thm:star}
Assume \eqref{eq:H} and let $(w_n)\subset X$ satisfy $w_n\rightharpoonup0$ in $X$.
Then
\begin{equation}\label{eq:star}
\liminf_{n\to\infty}\Big(Q_{\boldsymbol\mu,\mathbf a}(w_n)
-S_*\,\|w_n\|^2_{L^{\cstar}(\Omega)}\Big)\ \ge\ 0 .
\end{equation}
\end{theorem}

In Subsection~\ref{sec:dicho} we prove Theorem~\ref{thm:star} by localizing the
Gagliardo seminorm with a suitable partition of unity. The resulting additional terms
are controlled by the $L^2$ norm. We neither use a profile decomposition nor the
Caffarelli-Silvestre extension. It follows that the critical norm associated with a
Palais-Smale sequence is either $0$ or bounded below by $S_*^{N/(2s)}$. This gives
relative compactness of such sequences below $c_*$, and Theorem~\ref{thm:main} then
follows by direct minimization.

The paper is organised as follows. Section~\ref{sec:notation} fixes the functional
setting and collects the Hardy and Sobolev inequalities, the single-pole facts and the
strong minimum principle used throughout. Section~\ref{sec:framework} develops the
multipolar variational framework, the multipolar quadratic form is the square of a
Hilbert norm on $X$ equivalent to the Gagliardo seminorm
(Proposition~\ref{prop:form}). The quantity $\lambda_1(\boldsymbol\mu,\mathbf a)$ is
positive, attained and simple with eigenfunctions of constant sign
(Theorem~\ref{thm:A}). The map $\mu\mapsto S_\mu$ is strictly decreasing, so that
the compactness constant is $S_*=S_{\mu_{\max}}$ (Proposition~\ref{prop:Smu}).
Section~\ref{sec:asymptotics} analyses the homogeneity function $\Psins$, gives an
explicit formula for $\mubar$ and derives the scaling, tail and weighted
$L^2$-asymptotics for the single-pole extremals.
Section~\ref{sec:concentrating} establishes the expansion used in
\eqref{eq:Bjintro} and the resulting geometric threshold.
Section~\ref{sec:ps} proves the compactness of Palais-Smale sequences below $c_*$
from the dichotomy in \eqref{eq:star}, and Section~\ref{sec:existence} deduces
attainment and Theorem~\ref{thm:main}.

\smallskip
\subsection{Relation to previous work and outline of the proof}\label{sec:outline}
The two conclusions above should be compared with the following. Every existence result for \eqref{eq:P} with $k\le1$ requires $\lambda>0$. This is the
case in \cite{BN} for $s=1$, in \cite{SV-BN,SV-low} for the fractional problem without
Hardy terms, and in \cite{Ja,GRSZ} for a single pole, where the relevant coefficient is
$B_j=\lambda$ and no interaction is available. The threshold
$\lambda_{\mathrm{geo}}<0$ of Theorem~\ref{thm:main} therefore has no counterpart in
that literature. In the local multipolar problem \cite{FT,CH} the interaction between
poles is present, but the results there are conditions for attainment of a Rayleigh
quotient rather than a sign threshold on the linear term. Remark~\ref{rem:GRSZsection6}
isolates the mechanism, the linear term and the Hardy terms centred at the remaining
poles combine near $a_j$ into a single weight whose value at $a_j$ is $B_j$.

The arguments leading to these results are largely based on the single-pole theory. We indicate the connections with the existing results whenever relevant. The threshold in Theorem~\ref{thm:mubar} coincides with $\gamma_{\mathrm{crit}}(2s)$ introduced in \cite{GRSZ}. An explicit formula for this threshold is given in Remark~\ref{rem:gammacrit}. When $k=1$, the arguments of Subsection~\ref{sec:trunc} recover the cutoff expansion of \cite[Section~6]{GRSZ}.

For Theorem~\ref{thm:star}, we use a direct localization of the Gagliardo seminorm. This yields the multipolar constant
$$
S_*=\min{S,S_{\mu_1},\dots,S_{\mu_k}}
$$
on a bounded domain without using a profile decomposition or the Caffarelli--Silvestre extension. Localization arguments of this type as well as fractional versions of the underlying identity are known. In contrast, \cite{MMO} treats the whole-space multipolar problem through concentration-compactness in the extension formulation.

\section{Functional setting and preliminary results}\label{sec:notation}

This section fixes the function spaces and the bilinear forms used below, and
collects the standard results from the literature on which the rest of the paper
relies. These are stated in the notation and normalisation adopted here, proofs and
further details can be found in the cited references.

\subsection{Function spaces and variational formulation}\label{sec:spaces}
We begin with the function spaces and the two bilinear forms used below.
For a measurable function $u:\R^N\to\R$, set
\[
[u]_s^2:=\frac{C_{N,s}}{2}\iint_{\R^{2N}}\frac{|u(x)-u(y)|^2}{|x-y|^{N+2s}}\,dx\,dy
=\big\|(-\Delta)^{s/2}u\big\|_{L^2(\R^N)}^2 ,
\]
with the normalising constant $C_{N,s}$ of \cite{DNPV}. Let
$\dHs$ be the completion of $C^\infty_c(\R^N)$ with respect to $[\,\cdot\,]_s$, and
\[
X:=\big\{u\in \dHs\ :\ u=0 \text{ a.e. in }\R^N\setminus\Omega\big\},
\]
a Hilbert space with scalar product $\langle u,v\rangle_X=\frac{C_{N,s}}{2}\iint\frac{(u(x)-u(y))(v(x)-v(y))}{|x-y|^{N+2s}}\,dx\,dy$.
More generally we write, for $u,v\in\dHs$,
\[
\mathcal E(u,v):=\frac{C_{N,s}}{2}\iint_{\R^{2N}}
\frac{\big(u(x)-u(y)\big)\big(v(x)-v(y)\big)}{|x-y|^{N+2s}}\,dx\,dy ,
\]
so that $\mathcal E(u,u)=[u]_s^2$. The double integral converges absolutely, with
$|\mathcal E(u,v)|\le[u]_s[v]_s$ by the Cauchy-Schwarz inequality on
$\R^{2N}$ and $\mathcal E=\langle\cdot,\cdot\rangle_X$ on $X$. Define also, for $u,v\in X$,
\begin{equation}\label{eq:B}
\mathcal B_{\boldsymbol\mu,\mathbf a}(u,v)
:=\mathcal E(u,v)-\sum_{i=1}^{k}\mu_i\int_\Omega\frac{uv}{|x-a_i|^{2s}}\,dx ,
\end{equation}
a symmetric bilinear form with
$\mathcal B_{\boldsymbol\mu,\mathbf a}(u,u)=Q_{\boldsymbol\mu,\mathbf a}(u)$ by
\eqref{eq:Q}. The integrals converge for all $u,v\in X$ by
Proposition~\ref{prop:form}.
Since $\Omega$ is bounded and smooth, $X\hookrightarrow L^q(\Omega)$ continuously for
$q\in[1,\cstar]$ and compactly for $q\in[1,\cstar)$ \cite{DNPV,SV-BN}.

We say that $u\in X$ is a weak solution of \eqref{eq:P} if
\[
\langle u,\varphi\rangle_X-\sum_i\mu_i\int_\Omega\frac{u\varphi}{|x-a_i|^{2s}}
=\lambda\int_\Omega u\varphi+\int_\Omega|u|^{\cstar-2}u\varphi \, \text{ for all }\varphi\in X .
\]
All integrals here are finite and this is part of Proposition~\ref{prop:form}.

For $0\le\mu<\Lam$ put
\begin{equation}\label{eq:Smu}
S_\mu:=\inf_{u\in\dHs\setminus\{0\}}
\frac{[u]_s^2-\mu\displaystyle\int_{\R^N}\frac{u^2}{|x|^{2s}}dx}
{\Big(\displaystyle\int_{\R^N}|u|^{\cstar}dx\Big)^{2/\cstar}}, \, S:=S_0 .
\end{equation}

\subsection{Hardy and Sobolev inequalities}\label{sec:hardysob}

The following inequalities define the constants $\Lam$ and $S$ used below.

\begin{theorem}[Sharp fractional Hardy inequality \cite{Herbst,FLS,FS}]\label{thm:hardy}
For every $u\in\dHs$,
\[
\Lam\int_{\R^N}\frac{u^2}{|x|^{2s}}\,dx\le [u]_s^2, \,
\Lam=2^{2s}\,\frac{\Gamma\!\left(\frac{N+2s}{4}\right)^2}
{\Gamma\!\left(\frac{N-2s}{4}\right)^2},
\]
Moreover, $\Lam$ is optimal and is not attained.
\end{theorem}

\begin{theorem}[Sharp fractional Sobolev inequality, \cite{CT,Lieb}]\label{thm:sobolev}
There is a sharp constant $S>0$ such that
\[
S\,\|u\|_{L^{\cstar}(\R^N)}^2\ \le\ [u]_s^2 ,\, u\in\dHs .
\]
Furthermore, equality is attained by the fractional Aubin-Talenti functions
$x\mapsto c\,(\varepsilon^2+|x-x_0|^2)^{-(N-2s)/2}$.
\end{theorem}

\subsection{The single-pole problem}\label{sec:singlepole}

The following statements collect the basic properties of the single-pole problem,
including the relevant homogeneity exponents and the extremals for $S_\mu$.

\begin{lemma}[Fractional Laplacian of homogeneous functions \cite{Herbst,FLS,GRSZ}]
\label{lem:homogeneous}
For $\gamma\in(0,N-2s)$ the function $|x|^{-\gamma}$ satisfies
$(-\Delta)^s|x|^{-\gamma}=\Psins(\gamma)|x|^{-\gamma-2s}$ in $\R^N\setminus\{0\}$,
where
\begin{equation}\label{eq:psi}
\Psins(\gamma)=2^{2s}\,
\frac{\Gamma\!\left(\frac{\gamma+2s}{2}\right)\Gamma\!\left(\frac{N-\gamma}{2}\right)}
{\Gamma\!\left(\frac{\gamma}{2}\right)\Gamma\!\left(\frac{N-2s-\gamma}{2}\right)} .
\end{equation}
\end{lemma}

\begin{proposition}[Single-pole extremals and their asymptotics \cite{DMPS,GS,Fall,GRSZ}]
\label{prop:Umu}
Let $0<\mu<\Lam$. Then $S_\mu$ is attained in $\dHs$ by a positive, radially
symmetric, radially non-increasing function $U_\mu$, which after normalisation
satisfies
\begin{equation}\label{eq:weakUmu}
\mathcal E(U_\mu,\psi)-\mu\int_{\R^N}\frac{U_\mu\psi}{|x|^{2s}}\,dx
=\int_{\R^N}U_\mu^{\cstar-1}\psi\,dx\, \text{for every }\psi\in\dHs .
\end{equation}
Taking $\psi=U_\mu$ in \eqref{eq:weakUmu} gives
$[U_\mu]_s^2-\mu\int\frac{U_\mu^2}{|x|^{2s}}=\|U_\mu\|^{\cstar}_{L^{\cstar}}$ and
hence
\begin{equation}\label{eq:SmuUmu}
S_\mu=\|U_\mu\|_{L^{\cstar}}^{\cstar-2} .
\end{equation}
Moreover there is $C=C(N,s,\mu)\ge1$ with
\begin{equation}\label{eq:twoside}
C^{-1}\min\big\{|y|^{-\beta_-(\mu)},|y|^{-\beta_+(\mu)}\big\}\le U_\mu(y)\le
C\min\big\{|y|^{-\beta_-(\mu)},|y|^{-\beta_+(\mu)}\big\},\, y\neq0,
\end{equation}
where $\beta_-(\mu)<\frac{N-2s}{2}<\beta_+(\mu)$ are the two roots in $(0,N-2s)$ of
$\Psins(\gamma)=\mu$ (see Proposition~\ref{prop:psi}). The behaviour at the
singularity is the local asymptotics of \cite{FF}.
\end{proposition}

\subsection{Strong minimum principle}\label{sec:smp}

We use the following strong minimum principle to prove strict positivity.

\begin{theorem}[Strong minimum principle \cite{DPQ}]\label{thm:smp}
Let $D\subset\R^N$ be a bounded open set, let $c\in L^1_{\mathrm{loc}}(D)$ satisfy
$c\le0$ a.e., and let $w\in\dHs$ vanish a.e.\ outside $D$ and satisfy $w\ge0$ a.e.\ in
$\R^N$ together with
\[
\mathcal E(w,\varphi)\ \ge\ \int_D c\,w\,\varphi\,dx \, \text{for every }\varphi\in\dHs\ \text{with }\varphi=0\text{ a.e.\ outside }D,\
\varphi\ge0 .
\]
Then either $w>0$ a.e.\ in $D$ or $w=0$ a.e.\ in $\R^N$. (The normalisation of the fractional
Laplacian in \cite{DPQ} differs from ours by a positive multiplicative constant, which
does not affect the sign condition on $c$. We do not need connectedness of $D$.)
\end{theorem}

\section{The multipolar variational framework}\label{sec:framework}

This section develops the variational framework attached to the multipolar form, its
coercivity, its principal eigenvalue, and the constant $S_*$ that governs the
compactness threshold.

\subsection{Coercivity of the multipolar quadratic form}\label{sec:form}

Assuming \eqref{eq:H}, the Hardy terms can be estimated separately, and the resulting
quadratic form is equivalent to the square of the Gagliardo seminorm.

\begin{proposition}\label{prop:form}
Assume \eqref{eq:H} and set $\sigma:=\Lam^{-1}\sum_{i=1}^k\mu_i\in[0,1)$. Then for
every $u\in X$ each integral in \eqref{eq:Q} is finite and
\begin{equation}\label{eq:equiv}
(1-\sigma)\,[u]_s^2\ \le\ Q_{\boldsymbol\mu,\mathbf a}(u)\ \le\ [u]_s^2 .
\end{equation}
In particular $\mathcal B_{\boldsymbol\mu,\mathbf a}$ of \eqref{eq:B} is an inner
product on $X$. The associated norm
$\|u\|_{\boldsymbol\mu,\mathbf a}:=Q_{\boldsymbol\mu,\mathbf a}(u)^{1/2}$ is
equivalent to $[\,\cdot\,]_s$, so that $(X,\mathcal B_{\boldsymbol\mu,\mathbf a})$ is
a Hilbert space with the same weak topology as $X$, and
$Q_{\boldsymbol\mu,\mathbf a}$ is sequentially weakly lower semicontinuous on $X$.
\end{proposition}

\begin{proof}
Fix $i$ and $u\in X\subset\dHs$. The seminorm $[\,\cdot\,]_s$ is invariant under
translations, so applying Theorem~\ref{thm:hardy} to $u(\cdot+a_i)$ gives
\begin{equation}\label{eq:hardy-i}
\int_{\R^N}\frac{u^2}{|x-a_i|^{2s}}\,dx\ \le\ \frac{1}{\Lam}\,[u]_s^2<\infty .
\end{equation}
Since $u=0$ a.e.\ outside $\Omega$, the integral over $\Omega$ equals the integral
over $\R^N$. Summing \eqref{eq:hardy-i} against the weights $\mu_i\ge0$,
\[
\sum_{i=1}^k\mu_i\int_\Omega\frac{u^2}{|x-a_i|^{2s}}\,dx
\le\frac{\sum_i\mu_i}{\Lam}\,[u]_s^2=\sigma\,[u]_s^2 ,
\]
which gives the left inequality in \eqref{eq:equiv}. The right inequality follows because all the subtracted terms are nonnegative.

For $u,\varphi\in X$, Cauchy-Schwarz and \eqref{eq:hardy-i} give
$\int_\Omega\frac{|u\varphi|}{|x-a_i|^{2s}}\le\Lam^{-1}[u]_s[\varphi]_s<\infty$, so
$\mathcal B_{\boldsymbol\mu,\mathbf a}$ is a well defined symmetric bilinear form on
$X$, bounded by $(1+\sigma)[u]_s[\varphi]_s$, and the weak formulation in
Subsection~\ref{sec:spaces} is well posed. By \eqref{eq:equiv} it is positive
definite, and $\|\cdot\|_{\boldsymbol\mu,\mathbf a}$ is a norm equivalent to
$[\,\cdot\,]_s$. The two equivalent Hilbert norms on the same vector space induce the same
weak topology. Finally $Q_{\boldsymbol\mu,\mathbf a}=\|\cdot\|^2_{\boldsymbol\mu,\mathbf a}$
is convex and strongly continuous on $X$, therefore sequentially weakly lower
semicontinuous.
\end{proof}

\begin{remark}
Condition \eqref{eq:H} is sufficient but not necessary for \eqref{eq:equiv}. It does
not use the mutual distances of the poles. Sharper conditions in the local case are
in \cite{FMT}. In the fractional case, necessary and sufficient conditions on the
configuration of the poles for positivity are obtained in \cite{FMO}.
\end{remark}

\subsection{The principal eigenvalue}\label{sec:eigen}

The principal eigenvalue below determines the range of $\lambda$ used throughout the
paper.

\begin{definition}
Assuming \eqref{eq:H}, define
$\lambda_1(\boldsymbol\mu,\mathbf a):=\inf_{u\in X\setminus\{0\}}
\frac{Q_{\boldsymbol\mu,\mathbf a}(u)}{\displaystyle\int_\Omega u^2\,dx}.$
\end{definition}

The following theorem gives the properties of $\lambda_1(\boldsymbol\mu,\mathbf a)$
needed below.

\begin{theorem}\label{thm:A}
Assume \eqref{eq:H} and write $\lambda_1=\lambda_1(\boldsymbol\mu,\mathbf a)$,
$\sigma=\Lam^{-1}\sum_i\mu_i$. Then we have
\begin{enumerate}
\item[(i)] $(1-\sigma)\,S\,|\Omega|^{-2s/N}\le\lambda_1<\infty$, in particular $\lambda_1>0$.
\item[(ii)] The infimum is attained.
\item[(iii)] If $u$ is a minimizer, then $|u|$ is a minimizer and $|u|>0$ a.e.\ in
$\Omega$. Furthermore, $u$ itself has a.e.\ constant sign, either $u>0$ a.e.\ in $\Omega$
or $u<0$ a.e.\ in $\Omega$.
\item[(iv)] $\lambda_1$ is simple: the set of minimizers together with $0$ is a one-dimensional subspace of $X$.
\item[(v)] For every $\lambda<\lambda_1$ there is $\theta=\theta(\lambda)>0$ with
\[
Q_{\boldsymbol\mu,\mathbf a}(u)-\lambda\int_\Omega u^2\,dx\ \ge\ \theta\,[u]_s^2 \, \text{for all }u\in X .
\]
\end{enumerate}
\end{theorem}

\begin{proof}
\emph{(i).} By Proposition~\ref{prop:form}, Theorem~\ref{thm:sobolev} and H\"older's
inequality on the bounded set $\Omega$,
\[
Q_{\boldsymbol\mu,\mathbf a}(u)\ge(1-\sigma)[u]_s^2\ge(1-\sigma)S\|u\|_{L^{\cstar}(\Omega)}^2
\ge(1-\sigma)S\,|\Omega|^{-2s/N}\|u\|_{L^2(\Omega)}^2 ,
\]
where we used $\|u\|_{L^2}\le|\Omega|^{\frac12-\frac1{\cstar}}\|u\|_{L^{\cstar}}$ and
$\frac12-\frac1{\cstar}=\frac{s}{N}$. Taking the infimum gives the lower bound.
Choosing any fixed $u\in C_c^\infty(\Omega)\setminus\{0\}$ and using
$Q_{\boldsymbol\mu,\mathbf a}(u)\le[u]_s^2<\infty$ gives $\lambda_1<\infty$.

\emph{(ii).} Let $(u_n)\subset X$ with $\|u_n\|_{L^2(\Omega)}=1$ and
$Q_{\boldsymbol\mu,\mathbf a}(u_n)\to\lambda_1$. By \eqref{eq:equiv},
$[u_n]_s^2\le(1-\sigma)^{-1}Q_{\boldsymbol\mu,\mathbf a}(u_n)$ is bounded, so up to a
subsequence $u_n\rightharpoonup u$ in $X$ and, by compactness of
$X\hookrightarrow L^2(\Omega)$, $u_n\to u$ in $L^2(\Omega)$. Hence
$\|u\|_{L^2(\Omega)}=1$ and $u\ne0$. By the weak lower semicontinuity in
Proposition~\ref{prop:form},
\[
Q_{\boldsymbol\mu,\mathbf a}(u)\le\liminf_n Q_{\boldsymbol\mu,\mathbf a}(u_n)=\lambda_1 ,
\]
while $Q_{\boldsymbol\mu,\mathbf a}(u)\ge\lambda_1\|u\|^2_{L^2(\Omega)}=\lambda_1$ by
definition of $\lambda_1$. So $u$ is a minimizer.

\emph{(iii).} Let $u$ be a minimizer, that is $u\in X\setminus\{0\}$ and
$Q_{\boldsymbol\mu,\mathbf a}(u)=\lambda_1\|u\|^2_{L^2(\Omega)}$, and put $w:=|u|\in X$.
Since $\big||u(x)|-|u(y)|\big|\le|u(x)-u(y)|$ pointwise we have $[w]_s\le[u]_s$, while
$w^2=u^2$, so the Hardy terms and $\|u\|_{L^2}$  remain unchanged. Hence
\[
\lambda_1\|u\|^2_{L^2(\Omega)}=\lambda_1\|w\|^2_{L^2(\Omega)}
\ \le\ Q_{\boldsymbol\mu,\mathbf a}(w)\ \le\ Q_{\boldsymbol\mu,\mathbf a}(u)
=\lambda_1\|u\|^2_{L^2(\Omega)},
\]
the first inequality being the definition of $\lambda_1$ applied to $w\neq0$. So $w$
is a minimizer, and subtracting the equal (and finite, by
Proposition~\ref{prop:form}) Hardy terms from the middle equality gives
\begin{equation}\label{eq:seminormeq}
[\,|u|\,]_s^2=[u]_s^2 .
\end{equation}

By Proposition~\ref{prop:form} the map $v\mapsto Q_{\boldsymbol\mu,\mathbf a}(v)$ is a
bounded quadratic form on $X$, so $v\mapsto
Q_{\boldsymbol\mu,\mathbf a}(v)\|v\|_{L^2(\Omega)}^{-2}$ is of class $C^1$ on
$X\setminus\{0\}$ and its derivative vanishes at every minimizer, so every minimizer
satisfies
\begin{equation}\label{eq:eigen}
\langle v,\varphi\rangle_X-\sum_i\mu_i\int_\Omega\frac{v\varphi}{|x-a_i|^{2s}}
=\lambda_1\int_\Omega v\varphi \, \text{for all }\varphi\in X
\end{equation}
Every minimizer therefore satisfies \eqref{eq:eigen}, all terms being finite by
Proposition~\ref{prop:form}, so we use this for $v=w$. If
$\varphi\in X$ with $\varphi\ge0$, then, since $w\ge0$, $\mu_i\ge0$ and $\lambda_1>0$
by (i), \eqref{eq:eigen} gives
\[
\mathcal E(w,\varphi)=\lambda_1\int_\Omega w\varphi
+\sum_i\mu_i\int_\Omega\frac{w\varphi}{|x-a_i|^{2s}}\ \ge\ 0 .
\]
Thus $w$ satisfies the hypotheses of Theorem~\ref{thm:smp} with $D=\Omega$ and
$c\equiv0$: either $w>0$ a.e.\ in $\Omega$ or $w=0$ a.e.\ in $\R^N$. Since $u\neq0$, only the first case is possible, so $|u|=w>0$ a.e.\ in
$\Omega$.

It remains to show that $u$ has constant sign. For all real numbers $\alpha,\beta$,
\[
(\alpha-\beta)^2-\big(|\alpha|-|\beta|\big)^2=2\big(|\alpha\beta|-\alpha\beta\big)\ \ge\ 0 ,
\]
so, applying this with $\alpha=u(x)$, $\beta=u(y)$ and using \eqref{eq:seminormeq},
\[
0=[u]_s^2-[\,|u|\,]_s^2
=C_{N,s}\iint_{\R^{2N}}\frac{|u(x)u(y)|-u(x)u(y)}{|x-y|^{N+2s}}\,dx\,dy ,
\]
an integral with nonnegative integrand. Hence $u(x)u(y)\ge0$ for a.e.\
$(x,y)\in\R^{2N}$. If both $\{u>0\}$ and $\{u<0\}$ had positive measure, then
$\{u>0\}\times\{u<0\}$ would have positive measure in $\R^{2N}$ and $u(x)u(y)<0$ on
it, a contradiction, so one of the two sets has Lebesgue measure zero. Since $|u|>0$ a.e.\ in
$\Omega$, either $u>0$ a.e.\ in $\Omega$ or $u<0$ a.e.\ in $\Omega$.

\emph{(iv).} Let $E\subset X$ be the set of $v\in X$ satisfying \eqref{eq:eigen}.
Then $E$ is a linear subspace, and every $v\in E\setminus\{0\}$ is a minimizer,
since taking $\varphi=v$ in \eqref{eq:eigen} gives
$Q_{\boldsymbol\mu,\mathbf a}(v)=\lambda_1\|v\|^2_{L^2(\Omega)}$. Suppose
$\dim E\ge2$ and pick $u,v\in E$ linearly independent. Both are minimizers, so by
\emph{(iii)} each of them has a.e.\ constant sign in $\Omega$ and $|u|,|v|>0$ a.e.\
there, since $\Omega$ is bounded, $u,v\in L^1(\Omega)$ and
$\big|\int_\Omega u\big|=\int_\Omega|u|>0$, and likewise for $v$. Put $t:=\left(\int_\Omega u\right)/\left(\int_\Omega v\right)$ and $z:=u-tv\in E$.
Then $z\neq0$ by linear independence, so $z$ is a minimizer and \emph{(iii)} gives
$\left|\int_\Omega z\right|=\int_\Omega|z|>0$. But $\int_\Omega z=0$ by the choice of
$t$, a contradiction. Hence $\dim E=1$.

\emph{(v).} If $\lambda\le0$ then, by \eqref{eq:equiv},
$Q_{\boldsymbol\mu,\mathbf a}(u)-\lambda\|u\|_{L^2}^2\ge Q_{\boldsymbol\mu,\mathbf a}(u)
\ge(1-\sigma)[u]_s^2$, so $\theta=1-\sigma$ works. If $0<\lambda<\lambda_1$ then
$\|u\|^2_{L^2}\le\lambda_1^{-1}Q_{\boldsymbol\mu,\mathbf a}(u)$ gives
\[
Q_{\boldsymbol\mu,\mathbf a}(u)-\lambda\|u\|^2_{L^2}
\ge\Big(1-\frac{\lambda}{\lambda_1}\Big)Q_{\boldsymbol\mu,\mathbf a}(u)
\ge\Big(1-\frac{\lambda}{\lambda_1}\Big)(1-\sigma)[u]_s^2 ,
\]
so we have $\theta=(1-\lambda/\lambda_1)(1-\sigma)>0$.
\end{proof}

\subsection{The one-pole constants and the compactness level}\label{sec:Smu}

The constant $S_*$ of \eqref{eq:Sstar} is defined as a minimum over $k+1$ numbers. We
now identify this minimum explicitly.

\begin{proposition}\label{prop:Smu}
For $0\le\mu<\Lam$ one has $S_\mu\ge\left(1-\frac{\mu}{\Lam}\right)S>0$, and the map
$\mu\mapsto S_\mu$ is strictly decreasing on $[0,\Lam)$. Consequently, under
\eqref{eq:H},
\[
S_*=S_{\mu_{\max}},\,\text{and}\, S_{\mu_j}=S_*\ \text{ exactly for } j\in J_* .
\]
\end{proposition}

\begin{proof}
The lower bound follows from Theorems~\ref{thm:hardy} and~\ref{thm:sobolev} for
$u\in\dHs\setminus\{0\}$,
\[
[u]_s^2-\mu\int\frac{u^2}{|x|^{2s}}\ \ge\ \Big(1-\frac{\mu}{\Lam}\Big)[u]_s^2
\ \ge\ \Big(1-\frac{\mu}{\Lam}\Big)S\,\|u\|^2_{L^{\cstar}} .
\]
Let $0\le\mu_1<\mu_2<\Lam$. By Theorem~\ref{thm:sobolev} (if $\mu_1=0$) or
Proposition~\ref{prop:Umu} (if $\mu_1>0$), $S_{\mu_1}$ is attained by some
$V\in\dHs\setminus\{0\}$. By Theorem~\ref{thm:hardy},
$0<\int_{\R^N}\frac{V^2}{|x|^{2s}}dx<\infty$, the upper bound by Hardy and the lower
bound because $V\not\equiv0$. Using $V$ as a test function in \eqref{eq:Smu},
\[
S_{\mu_2}\ \le\
\frac{[V]_s^2-\mu_2\int\frac{V^2}{|x|^{2s}}}{\|V\|_{L^{\cstar}}^2}
= S_{\mu_1}-(\mu_2-\mu_1)\,\frac{\int\frac{V^2}{|x|^{2s}}}{\|V\|_{L^{\cstar}}^2}
\ <\ S_{\mu_1} .
\]
Strict monotonicity gives $\min\{S,S_{\mu_1},\dots,S_{\mu_k}\}=S_{\mu_{\max}}$ (note
$S=S_0\ge S_{\mu_{\max}}$) and $S_{\mu_j}=S_{\mu_{\max}}$ if and only if
$\mu_j=\mu_{\max}$.
\end{proof}

\section{Single-pole asymptotics and perturbation estimates}\label{sec:asymptotics}

This section collects the asymptotic information about the single-pole extremals
$U_\mu$ that drives the expansion of Section~\ref{sec:concentrating}, the homogeneity
function and its threshold, the scaling and tail estimates, and the resulting
weighted $L^2$-asymptotics.

\subsection{The homogeneity function}\label{sec:psi}

We analyse the function $\Psins$ of Lemma~\ref{lem:homogeneous} and determine when the
one-pole extremal decays too slowly to belong to $L^2$.

\begin{proposition}\label{prop:psi}
Let $N>2s$ and let $\Psins$ be as in Lemma~\ref{lem:homogeneous}. Then we have
\begin{enumerate}
\item[(i)] $\Psins$ is positive and real analytic on $(0,N-2s)$, and
$\Psins(\gamma)=\Psins(N-2s-\gamma)$ for all $\gamma\in(0,N-2s)$.
\item[(ii)] $\Psins$ is strictly increasing on $\left(0,\frac{N-2s}{2}\right]$ and
strictly decreasing on $\left[\frac{N-2s}{2},N-2s\right)$, with
\[
\max_{(0,N-2s)}\Psins=\Psins\!\Big(\frac{N-2s}{2}\Big)=\Lam ,
\]
and $\Psins(\gamma)\to0$ as $\gamma\to0^+$ and as $\gamma\to(N-2s)^-$.
\item[(iii)] for every $\mu\in(0,\Lam)$ the equation $\Psins(\gamma)=\mu$ has exactly
two solutions $\beta_-(\mu)<\frac{N-2s}{2}<\beta_+(\mu)$ in $(0,N-2s)$, and
$\beta_-(\mu)+\beta_+(\mu)=N-2s$. With the notation
\[
\beta_+(0):=N-2s=\lim_{\mu\to0^+}\beta_+(\mu),\,
\beta_-(0):=0=\lim_{\mu\to0^+}\beta_-(\mu),
\]
the map $\mu\mapsto\beta_+(\mu)$ is strictly decreasing on $[0,\Lam)$ and
$\beta_+(\mu)\to\frac{N-2s}{2}$ as $\mu\to\Lam^-$.
\end{enumerate}
\end{proposition}

\begin{proof}
\emph{(i).} For $\gamma\in(0,N-2s)$ the four arguments
$\frac{\gamma+2s}{2},\ \frac{N-\gamma}{2},\ \frac{\gamma}{2},\ \frac{N-2s-\gamma}{2}$
are all positive, so $\Psins$ is positive and real analytic there. Replacing $\gamma$
by $N-2s-\gamma$ maps the pair of numerator arguments
$\left(\frac{\gamma+2s}{2},\frac{N-\gamma}{2}\right)$ onto
$\left(\frac{N-\gamma}{2},\frac{\gamma+2s}{2}\right)$ and the pair of denominator
arguments $\left(\frac{\gamma}{2},\frac{N-2s-\gamma}{2}\right)$ onto
$\left(\frac{N-2s-\gamma}{2},\frac{\gamma}{2}\right)$, so the quotient is unchanged.

\emph{(ii).} Set
\[
G(t):=\log\Gamma(t+s)-\log\Gamma(t),\, t>0 .
\]
Using $\frac{\gamma+2s}{2}=\frac{\gamma}{2}+s$ and
$\frac{N-\gamma}{2}=\frac{N-2s-\gamma}{2}+s$, formula \eqref{eq:psi} reads
\begin{equation}\label{eq:logpsi}
\log\Psins(\gamma)=2s\log2+G\Big(\frac{\gamma}{2}\Big)
+G\Big(\frac{N-2s-\gamma}{2}\Big),\, \gamma\in(0,N-2s).
\end{equation}
(The symmetry in (i) is true since $\gamma\mapsto N-2s-\gamma$ exchanges the two
arguments.) Let $\psi=\Gamma'/\Gamma$ and $g:=G'$, that is
$g(t)=\psi(t+s)-\psi(t)$. Differentiating \eqref{eq:logpsi} in $\gamma$, the second
term depending on $\gamma$ using $\frac{N-2s-\gamma}{2}$, whose derivative is
$-\frac12$,
\[
\frac{d}{d\gamma}\log\Psins(\gamma)
=\frac12\,g\Big(\frac{\gamma}{2}\Big)-\frac12\,g\Big(\frac{N-2s-\gamma}{2}\Big).
\]
Since $\psi'(t)=\sum_{n\ge0}(t+n)^{-2}$ is strictly decreasing on $(0,\infty)$, we
have $g'(t)=\psi'(t+s)-\psi'(t)<0$, so $g$ is strictly decreasing. Therefore, we have,
\[
\frac{d}{d\gamma}\log\Psins(\gamma)>0
\iff g\Big(\frac{\gamma}{2}\Big)>g\Big(\frac{N-2s-\gamma}{2}\Big)
\iff \frac{\gamma}{2}<\frac{N-2s-\gamma}{2}
\iff \gamma<\frac{N-2s}{2},
\]
and the reversed strict inequalities hold for $\gamma>\frac{N-2s}{2}$. Evaluating
\eqref{eq:psi} at $\gamma=\frac{N-2s}{2}$ gives
$\Psins\big(\frac{N-2s}{2}\big)=2^{2s}\Gamma\big(\frac{N+2s}{4}\big)^2
\Gamma\big(\frac{N-2s}{4}\big)^{-2}=\Lam$. The two limits follow from
$\Gamma(\tau)\to+\infty$ as $\tau\to0^+$ applied to $\Gamma(\gamma/2)$ and to
$\Gamma\big(\frac{N-2s-\gamma}{2}\big)$ respectively, the other three factors staying
in a compact subset of $(0,\infty)$.

\emph{(iii).} By (ii) and the intermediate value theorem, $\Psins$ is a strictly
increasing bijection from $\big(0,\frac{N-2s}{2}\big]$ onto $(0,\Lam]$ and a strictly
decreasing bijection from $\big[\frac{N-2s}{2},N-2s\big)$ onto $(0,\Lam]$, so for
$\mu\in(0,\Lam)$ the level set $\{\Psins=\mu\}$ consists of exactly one point in each
of the two intervals. If $\Psins(\beta_+)=\mu$ then $\Psins(N-2s-\beta_+)=\mu$ by (i)
and $N-2s-\beta_+<\frac{N-2s}{2}$, so $N-2s-\beta_+=\beta_-$. The map $\beta_+$ is the
inverse of the strictly decreasing bijection above, hence strictly decreasing and
continuous with $\beta_+(\mu)\to N-2s$ as $\mu\to0^+$ and
$\beta_+(\mu)\to\frac{N-2s}{2}$ as $\mu\to\Lam^-$.
\end{proof}

\subsection{The threshold $\mubar$ and $L^2$-integrability}\label{sec:mubar}

The monotonicity properties of $\Psins$ give an explicit equivalent condition for
$\beta_+(\mu)>\frac N2$.

\begin{theorem}\label{thm:mubar}
Let $\mu\in[0,\Lam)$ and let $\beta_+(\mu)$ be as in
Proposition~\ref{prop:psi}. Then we have
\begin{enumerate}
\item[(i)] If $N\le 4s$, then $\beta_+(\mu)\le\frac{N}{2}$ for every
$\mu\in[0,\Lam)$, with equality only when $N=4s$ and $\mu=0$.
\item[(ii)] If $N>4s$, then
\[
\beta_+(\mu)>\frac{N}{2}\quad\Longleftrightarrow\quad \mu<\mubar,\,
\mubar:=\Psins\Big(\frac{N}{2}\Big)
=2^{2s}\,\frac{\Gamma\!\left(\frac{N+4s}{4}\right)}{\Gamma\!\left(\frac{N-4s}{4}\right)} ,
\]
and $0<\mubar<\Lam$.
\end{enumerate}
\end{theorem}

\begin{proof}
\emph{(i).} If $N\le4s$ then $N-2s\le\frac N2$. Since
$\beta_+(\mu)\in\big[\frac{N-2s}{2},N-2s\big)$ for $\mu\in(0,\Lam)$ and
$\beta_+(0)=N-2s$, we get $\beta_+(\mu)\le N-2s\le\frac N2$, with equality throughout
only if $\mu=0$ and $N=4s$.

\emph{(ii).} Let $N>4s$. Then $\frac{N-2s}{2}<\frac N2<N-2s$, so $\frac N2$ lies in
the interval on which $\Psins$ is strictly decreasing and
$\mubar=\Psins(N/2)$ is well defined. Evaluating \eqref{eq:psi} at $\gamma=N/2$,
\[
\Psins\Big(\frac N2\Big)
=2^{2s}\frac{\Gamma\!\left(\frac{N+4s}{4}\right)\Gamma\!\left(\frac{N}{4}\right)}
{\Gamma\!\left(\frac{N}{4}\right)\Gamma\!\left(\frac{N-4s}{4}\right)}
=2^{2s}\frac{\Gamma\!\left(\frac{N+4s}{4}\right)}{\Gamma\!\left(\frac{N-4s}{4}\right)} ,
\]
all four arguments being positive since $N>4s$. Since
$\beta_+(\mu)\ge\frac{N-2s}{2}$ and $\Psins$ is strictly decreasing on
$\big[\frac{N-2s}{2},N-2s\big)$,
\[
\beta_+(\mu)>\frac N2
\iff \Psins\big(\beta_+(\mu)\big)<\Psins\Big(\frac N2\Big)
\iff \mu<\mubar .
\]
(For $\mu=0$ read $\beta_+(0)=N-2s>\frac N2$ and $0<\mubar$.) Finally
$\frac{N-2s}{2}<\frac N2$ and strict monotonicity give
$\mubar=\Psins(N/2)<\Psins\big(\frac{N-2s}{2}\big)=\Lam$ and $\mubar>0$ by
Proposition~\ref{prop:psi}(i).
\end{proof}

\begin{remark}[consistency with the local case]\label{rem:local}
The right-hand sides of \eqref{eq:psi} and of the formula for $\mubar$ are defined for
every $s>0$ with $N>4s$. Setting $s=1$ gives
$\Psi_{N,1}(\gamma)=\gamma(N-2-\gamma)$, the homogeneity function of $-\Delta$, and
\[
\bar\mu_{N,1}=\frac{N(N-4)}{4}=\frac{(N-2)^2}{4}-1 ,
\]
the constant below which Jannelli \cite{Ja} obtained extremals for every
$0<\lambda<\lambda_1(\mu)$ and the condition $N>4s$ specialises to $N>4$.
\end{remark}

\begin{remark}[relation with the threshold of \cite{GRSZ}]\label{rem:gammacrit}
The threshold $\mubar$ coincides with the constant $\gamma_{\mathrm{crit}}(2s)$
introduced in \cite{GRSZ}, characterised there by
$\beta_+(\gamma_{\mathrm{crit}})=\frac N2$. Theorem~\ref{thm:mubar} supplies the
closed form and the comparison $\mubar<\Lam$.
\end{remark}

The preceding result yields the following $L^2$-integrability criterion.

\begin{corollary}\label{cor:L2}
Let $0<\mu<\Lam$ and let $U_\mu$ be as in Proposition~\ref{prop:Umu}. Then
$U_\mu\in L^2(\R^N)$ if and only if $N>4s$ and $\mu<\mubar$.
\end{corollary}

\begin{proof}
By \eqref{eq:twoside} and polar coordinates,
\[
\int_{|y|\le1}U_\mu^2\,dy\ \asymp\ \int_0^1 r^{N-1-2\beta_-}dr, \,
\int_{|y|\ge1}U_\mu^2\,dy\ \asymp\ \int_1^\infty r^{N-1-2\beta_+}dr ,
\]
with two-sided constants depending only on $N,s,\mu$. The first integral is finite
since $2\beta_-<N-2s<N$ by Proposition~\ref{prop:psi}. The second is finite if and
only if $2\beta_+>N$, which by Theorem~\ref{thm:mubar} happens exactly when $N>4s$
and $\mu<\mubar$. (Here $A\asymp B$ means that $cB\leq A\leq CB$ for some constants
$c,C>0$ depending only on $N,s,\mu$.)
\end{proof}

\subsection{Scaling and tail estimates}\label{sec:order}

The following estimates describe the scaling invariances and the tail behaviour of
$U_\mu$ needed in the subsequent calculations.

Throughout this section $0<\mu<\Lam$, $U:=U_\mu$ is as in Proposition~\ref{prop:Umu},
$\beta:=\beta_+(\mu)$, and
\[
U_\varepsilon(x):=\varepsilon^{-\frac{N-2s}{2}}\,U\!\left(\frac{x}{\varepsilon}\right), \, \varepsilon>0 .
\]

\begin{lemma}\label{lem:scaling}
For every $\varepsilon>0$,
\[
[U_\varepsilon]_s^2=[U]_s^2,\,
\int_{\R^N}\frac{U_\varepsilon^2}{|x|^{2s}}\,dx=\int_{\R^N}\frac{U^2}{|y|^{2s}}\,dy,
\,
\int_{\R^N}|U_\varepsilon|^{\cstar}dx=\int_{\R^N}|U|^{\cstar}dy .
\]
If moreover $U\in L^2(\R^N)$, then
$\displaystyle\int_{\R^N}U_\varepsilon^2\,dx=\varepsilon^{2s}\int_{\R^N}U^2\,dy$.
\end{lemma}

\begin{proof}
All four identities follow from the substitution $x=\varepsilon y$. For the second,
\[
\int\frac{U_\varepsilon(x)^2}{|x|^{2s}}dx
=\varepsilon^{-(N-2s)}\int\frac{U(x/\varepsilon)^2}{|x|^{2s}}dx
=\varepsilon^{-(N-2s)}\varepsilon^{N-2s}\int\frac{U(y)^2}{|y|^{2s}}dy .
\]
For the fourth, $\varepsilon^{-(N-2s)}\varepsilon^{N}=\varepsilon^{2s}$. The first and
third are the standard invariances of $[\,\cdot\,]_s$ and of the $L^{\cstar}$ norm
under this scaling.
\end{proof}

Lemma~\ref{lem:scaling} displays the relevant scaling: the Hardy term at the
concentrating pole and the critical term are invariant, whereas $L^2$-type quantities
scale as $\varepsilon^{2s}$.

\begin{proposition}\label{prop:tails}
There is $C=C(N,s,\mu)\ge1$ such that for all $R\ge1$,
\begin{enumerate}
\item[(a)] if $2\beta>N$, then $U\in L^2(\R^N)$ and
$C^{-1}R^{N-2\beta}\le\displaystyle\int_{|y|>R}U^2\,dy\le C\,R^{N-2\beta}$,
\item[(b)] if $2\beta>N+1$, then $\displaystyle\int_{|y|\le R}|y|\,U^2\,dy\le C$,
\item[(c)] if $2\beta=N+1$, then $\displaystyle\int_{|y|\le R}|y|\,U^2\,dy\le C(1+\log R)$,
\item[(d)] if $N<2\beta<N+1$, then $\displaystyle\int_{|y|\le R}|y|\,U^2\,dy\le C\,R^{N+1-2\beta}$,
\item[(e)] $\displaystyle\int_{|y|>R}U^{\cstar}dy\le C\,R^{-\tau}$ with
$\tau:=\frac{2N\beta}{N-2s}-N>0$, if in addition $2\beta>N$, then
$\tau>\frac{2sN}{N-2s}>2s$.
\end{enumerate}
\end{proposition}

\begin{proof}
We write $\beta_-=N-2s-\beta$ and use \eqref{eq:twoside} together with polar coordinates,
all constants below depend only on $N,s,\mu$ using the constant $C$ of
\eqref{eq:twoside} and on the surface measure $\omega_{N-1}$ of $S^{N-1}$.

\emph{(a).} $\int_{|y|>R}U^2\,dy\le C^2\omega_{N-1}\int_R^\infty r^{N-1-2\beta}dr
=\frac{C^2\omega_{N-1}}{2\beta-N}R^{N-2\beta}$, finite since $2\beta>N$. The lower
bound is the same computation with $C^{-2}$. Integrability near the origin was shown
in the proof of Corollary~\ref{cor:L2}.

\emph{(b)-(d).} We split at $|y|=1$. On $\{|y|\le1\}$,
$\int_{|y|\le1}|y|U^2\le C^2\omega_{N-1}\int_0^1r^{N-2\beta_-}dr<\infty$ since
$2\beta_-<N-2s<N+1$. On $\{1\le|y|\le R\}$,
\[
\int_{1\le|y|\le R}|y|U^2\,dy\le C^2\omega_{N-1}\int_1^R r^{N-2\beta}dr,
\]
which equals $\frac{C^2\omega_{N-1}}{2\beta-N-1}\left(1-R^{N+1-2\beta}\right)$ if
$2\beta>N+1$, equals $C^2\omega_{N-1}\log R$ if $2\beta=N+1$, and is at most
$\frac{C^2\omega_{N-1}}{N+1-2\beta}R^{N+1-2\beta}$ if $2\beta<N+1$.

\emph{(e).} By \eqref{eq:twoside},
$\int_{|y|>R}U^{\cstar}dy\le C^{\cstar}\omega_{N-1}\int_R^\infty
r^{N-1-\cstar\beta}dr$. Convergence uses only the bound
$\beta>\frac{N-2s}{2}$ of Proposition~\ref{prop:psi}(iii),
\[
\cstar\beta=\frac{2N\beta}{N-2s}>\frac{2N}{N-2s}\cdot\frac{N-2s}{2}=N ,
\]
that is $\tau=\cstar\beta-N>0$, and then
$\int_R^\infty r^{N-1-\cstar\beta}dr=\tau^{-1}R^{-\tau}$.

Assume in addition that $2\beta>N$. Then
\[
\tau=\frac{2N\beta}{N-2s}-N>\frac{2N}{N-2s}\cdot\frac{N}{2}-N
=\frac{N^2-N(N-2s)}{N-2s}=\frac{2sN}{N-2s} ,
\]
and $\frac{2sN}{N-2s}>2s$ since $\frac{N}{N-2s}>1$.

\end{proof}

\subsection{Weighted $L^2$-asymptotics}\label{sec:weighted}

The tail estimates yield the following asymptotic estimate for integrals of
$U_\varepsilon^2$ against Lipschitz weights.

\begin{theorem}\label{thm:order}
Let $0<\mu<\Lam$ with $\beta=\beta_+(\mu)>\frac N2$, let $\rho>0$, and let
$W:\overline{B_\rho(0)}\to\R$ be Lipschitz with constant $L$. Then, as
$\varepsilon\to0^+$,
\begin{equation}\label{eq:order}
\int_{B_\rho(0)}W(x)\,U_\varepsilon(x)^2\,dx
=\varepsilon^{2s}\Big(W(0)\,\|U\|_{L^2(\R^N)}^2+E(\varepsilon)\Big),
\end{equation}
where, with $C=C(N,s,\mu)$ and $\rho$ fixed,
\[
|E(\varepsilon)|\ \le\ C\Big(|W(0)|\,\varepsilon^{2\beta-N}
+L\,\varepsilon\,h_\varepsilon\Big), \,
h_\varepsilon:=
\begin{cases}
1, & 2\beta>N+1,\\[1mm]
1+\log\frac1\varepsilon, & 2\beta=N+1,\\[1mm]
\varepsilon^{\,2\beta-N-1}, & N<2\beta<N+1 .
\end{cases}
\]
In particular $E(\varepsilon)\to 0$, and
$E(\varepsilon)=O\big(\varepsilon^{\kappa}\log\frac1\varepsilon\big)$ with
$\kappa=\min\{1,2\beta-N\}>0$.
\end{theorem}

\begin{proof}
By Corollary~\ref{cor:L2}, $U\in L^2(\R^N)$. Substitute $x=\varepsilon y$ and use the
identity $U_\varepsilon(\varepsilon y)^2=\varepsilon^{-(N-2s)}U(y)^2$, so that
\[
\int_{B_\rho(0)}W(x)U_\varepsilon(x)^2dx
=\varepsilon^{2s}\int_{B_{\rho/\varepsilon}(0)}W(\varepsilon y)\,U(y)^2\,dy .
\]
Split the right-hand integral and we get,
\[
\int_{B_{\rho/\varepsilon}}W(\varepsilon y)U^2
=W(0)\int_{\R^N}U^2
-W(0)\int_{|y|>\rho/\varepsilon}U^2
+\int_{B_{\rho/\varepsilon}}\big(W(\varepsilon y)-W(0)\big)U^2 ,
\]
so that $E(\varepsilon)=-W(0)\int_{|y|>\rho/\varepsilon}U^2
+\int_{B_{\rho/\varepsilon}}(W(\varepsilon y)-W(0))U^2$.

For the first term, Proposition~\ref{prop:tails}(a) with $R=\rho/\varepsilon\ge1$
(true for $\varepsilon\le\rho$) gives
$\int_{|y|>\rho/\varepsilon}U^2\le C\rho^{N-2\beta}\varepsilon^{2\beta-N}$.

For the second term, $|W(\varepsilon y)-W(0)|\le L\varepsilon|y|$ for
$|y|\le\rho/\varepsilon$, so
\[
\Big|\int_{B_{\rho/\varepsilon}}(W(\varepsilon y)-W(0))U^2\Big|
\le L\,\varepsilon\int_{|y|\le\rho/\varepsilon}|y|\,U(y)^2dy ,
\]
and Proposition~\ref{prop:tails}(b)-(d) with $R=\rho/\varepsilon$ bounds the last
integral by $C$, by $C(1+\log\frac{\rho}{\varepsilon})$, or by
$C(\rho/\varepsilon)^{N+1-2\beta}$ in the three cases, which is the asserted
$h_\varepsilon$ up to constants depending on $\rho$. Since $2\beta>N$ and
$\varepsilon h_\varepsilon\le C\varepsilon^{\min\{1,2\beta-N\}}(1+\log\frac1\varepsilon)$
in all three cases, the last assertion follows.
\end{proof}

Applying Theorem~\ref{thm:order} to the effective weight near $a_j$ gives the
following.

\begin{corollary}\label{cor:Bj}
Assume \eqref{eq:H}, let $j\in\{1,\dots,k\}$ with $\mu_j>0$ and
$\beta_j:=\beta_+(\mu_j)>\frac N2$, and set
\[
U_{\varepsilon,j}(x):=\varepsilon^{-\frac{N-2s}{2}}U_{\mu_j}\!\Big(\frac{x-a_j}{\varepsilon}\Big),
\,
B_j(\lambda,\mathbf a,\boldsymbol\mu):=\lambda+\sum_{i\neq j}\frac{\mu_i}{|a_i-a_j|^{2s}} .
\]
Fix $0<\rho<\frac{d_j}{2}$. Then, as $\varepsilon\to0^+$,
\[
\int_{B_\rho(a_j)}\Big(\lambda+\sum_{i\neq j}\frac{\mu_i}{|x-a_i|^{2s}}\Big)
U_{\varepsilon,j}(x)^2\,dx
=B_j(\lambda,\mathbf a,\boldsymbol\mu)\,\|U_{\mu_j}\|_{L^2(\R^N)}^2\,\varepsilon^{2s}
+O\Big(\varepsilon^{2s+\kappa_j}\log\tfrac1\varepsilon\Big),
\]
with $\kappa_j=\min\{1,2\beta_j-N\}>0$. Moreover, by Lemma~\ref{lem:scaling},
\[
[U_{\varepsilon,j}]_s^2-\mu_j\int_{\R^N}\frac{U_{\varepsilon,j}^2}{|x-a_j|^{2s}}dx
=[U_{\mu_j}]_s^2-\mu_j\int_{\R^N}\frac{U_{\mu_j}^2}{|y|^{2s}}dy
=S_{\mu_j}\,\|U_{\mu_j}\|_{L^{\cstar}}^{2},
\]
for every $\varepsilon>0$.
\end{corollary}

\begin{proof}
Set $W(x):=\lambda+\sum_{i\neq j}\mu_i|x-a_i|^{-2s}$ for
$x\in\overline{B_\rho(a_j)}$. Since $\rho<\frac{d_j}{2}$, we have
$|x-a_i|\ge|a_i-a_j|-\rho\ge d_j-\rho>\frac{d_j}{2}$ for every $i\neq j$ and every
$x\in\overline{B_\rho(a_j)}$, so $W$ is smooth on $\overline{B_\rho(a_j)}$, in
particular Lipschitz there, and $W(a_j)=B_j(\lambda,\mathbf a,\boldsymbol\mu)$. Apply
Theorem~\ref{thm:order} after the translation $x\mapsto x-a_j$. The last part follows from
Lemma~\ref{lem:scaling} together with the fact that $U_{\mu_j}$ attains $S_{\mu_j}$.
\end{proof}

\section{Concentrating test functions and the geometric threshold}%
\label{sec:concentrating}
We now truncate the concentrating family so that it belongs to $X$. The error introduced by the cutoff is of lower order than the leading term in Corollary~\ref{cor}. This yields the expansion of the Rayleigh quotient along the concentrating family and the corresponding geometric threshold $\lambda_{\mathrm{geo}}$.

Throughout this section we fix $j\in\{1,\dots,k\}$ with
\[
\mu_j\in(0,\Lam),\, \beta_j:=\beta_+(\mu_j)>\frac N2 ,
\]
we abbreviate $U^{(j)}:=U_{\mu_j}$ and set
\[
U_{\varepsilon,j}(x):=\varepsilon^{-\frac{N-2s}{2}}\,
U^{(j)}\!\Big(\frac{x-a_j}{\varepsilon}\Big),\, \varepsilon>0 ,
\]
and we fix a radius $\rho$ and a cutoff $\eta$ with
\begin{equation}\label{eq:cutoff}
0<\rho<\min\Big\{\frac{d_j}{2},\ \operatorname{dist}(a_j,\partial\Omega)\Big\},\,
\eta\in C^\infty_c\big(B_\rho(a_j)\big),\quad 0\le\eta\le1,\quad
\eta\equiv1\ \text{on}\ B_{\rho/2}(a_j) .
\end{equation}
By Corollary~\ref{cor:L2}, $U^{(j)}\in L^2(\R^N)$, and by Lemma~\ref{lem:scaling},
$U_{\varepsilon,j}\in\dHs\cap L^2(\R^N)$ for every $\varepsilon>0$.

\subsection{Cutoff estimates and the cutoff identity}\label{sec:cutoff}

We first record the elementary cutoff bounds and the exact localization identity for
the Gagliardo seminorm.

\begin{lemma}\label{lem:elem}
Let $\theta:\R^N\to[0,1]$ be Lipschitz with constant $L>0$. Then
\[
\sup_{y\in\R^N}\int_{\R^N}\frac{\big(\theta(x)-\theta(y)\big)^2}{|x-y|^{N+2s}}\,dx
\ \le\ c_{N,s}\,L^{2s},\,
c_{N,s}:=\omega_{N-1}\Big(\frac{1}{2-2s}+\frac{1}{2s}\Big),
\]
where $\omega_{N-1}=\mathcal H^{N-1}(S^{N-1})$.
\end{lemma}

\begin{proof}
Since $0\le\theta\le1$ and $\theta$ is $L$-Lipschitz,
$(\theta(x)-\theta(y))^2\le\min\{1,L^2|x-y|^2\}$. Passing to polar coordinates
centred at $y$ and writing $t=|x-y|$, $T=1/L$,
\[
\int_{\R^N}\frac{\min\{1,L^2|x-y|^2\}}{|x-y|^{N+2s}}dx
=\omega_{N-1}\Big(L^2\!\int_0^{T}\!t^{1-2s}dt+\int_T^\infty\! t^{-1-2s}dt\Big)
=\omega_{N-1}\Big(\frac{L^2T^{2-2s}}{2-2s}+\frac{T^{-2s}}{2s}\Big),
\]
both integrals converging since $0<s<1$. Now $L^2T^{2-2s}=T^{-2s}=L^{2s}$.
\end{proof}

The preceding estimate also gives the following cutoff bound.

\begin{lemma}\label{lem:product}
Let $\theta:\R^N\to[0,1]$ be Lipschitz with constant $L$ and let
$u\in\dHs\cap L^2(\R^N)$. Then $\theta u\in\dHs$ and
\[
[\theta u]_s^2\ \le\ 2\,[u]_s^2+C_{N,s}\,c_{N,s}\,L^{2s}\,\|u\|_{L^2(\R^N)}^2 .
\]
\end{lemma}

\begin{proof}
From $\theta(x)u(x)-\theta(y)u(y)=\theta(x)\big(u(x)-u(y)\big)+u(y)\big(\theta(x)-\theta(y)\big)$
and $(\alpha+\beta)^2\le2\alpha^2+2\beta^2$ we get, pointwise,
\[
\big(\theta(x)u(x)-\theta(y)u(y)\big)^2\le 2\big(u(x)-u(y)\big)^2
+2\,u(y)^2\big(\theta(x)-\theta(y)\big)^2 ,
\]
using $0\le\theta\le1$. Multiplying by $\tfrac{C_{N,s}}{2}|x-y|^{-N-2s}$, integrating
over $\R^{2N}$ and applying Tonelli's theorem and Lemma~\ref{lem:elem} to the second
term gives the stated bound which is finite. Hence $\theta u\in\dHs$.
\end{proof}

The rescaled family satisfies the corresponding translated equation.

\begin{lemma}\label{lem:eqscaled}
For every $\varepsilon>0$,
\begin{equation}\label{eq:weakUeps}
\mathcal E(U_{\varepsilon,j},\varphi)
-\mu_j\int_{\R^N}\frac{U_{\varepsilon,j}\,\varphi}{|x-a_j|^{2s}}\,dx
=\int_{\R^N}U_{\varepsilon,j}^{\cstar-1}\varphi\,dx \, \text{for every }\varphi\in\dHs .
\end{equation}
\end{lemma}

\begin{proof}
Let $T_\varepsilon\varphi(y):=\varepsilon^{\frac{N-2s}{2}}\varphi(a_j+\varepsilon y)$,
so that $T_\varepsilon U_{\varepsilon,j}=U^{(j)}$ and $T_\varepsilon$ is a linear
bijection of $\dHs$ with $[T_\varepsilon\varphi]_s=[\varphi]_s$, being a linear
isometry it preserves the bilinear form. Therefore,
$\mathcal E(U_{\varepsilon,j},\varphi)=\mathcal E(U^{(j)},T_\varepsilon\varphi)$.
For the remaining two integrals substitute $x=a_j+\varepsilon y$ and put
$\psi=T_\varepsilon\varphi$. The Hardy integrand picks up the factor
$\varepsilon^{-\frac{N-2s}{2}}\cdot\varepsilon^{-\frac{N-2s}{2}}\cdot
\varepsilon^{-2s}\cdot\varepsilon^{N}=1$, and the critical one the factor
$\varepsilon^{-\frac{N-2s}{2}\cstar}\varepsilon^{N}=1$ since
$\frac{N-2s}{2}\cstar=N$. Hence both sides of \eqref{eq:weakUeps} coincide with the
corresponding sides of \eqref{eq:weakUmu} evaluated at $\psi$, and
\eqref{eq:weakUmu} applies since $\psi\in\dHs$.
\end{proof}

The following identity avoids a direct comparison of $[\eta U_{\varepsilon,j}]_s$
with $[U_{\varepsilon,j}]_s$. It gives the quadratic form corresponding to the pole $a_j$ directly, so no separate estimate of the cutoff Hardy term is required.
\begin{lemma}[cutoff identity]\label{lem:identity}
Let $\theta:\R^N\to[0,1]$ be Lipschitz with constant $L$ and let
$u\in\dHs\cap L^2(\R^N)$. Then $\theta u,\ \theta^2u\in\dHs$ and
\begin{equation}\label{eq:identity}
[\theta u]_s^2=\mathcal E\big(u,\theta^2u\big)+R_\theta(u),\,
R_\theta(u):=\frac{C_{N,s}}{2}\iint_{\R^{2N}}
\frac{u(x)\,u(y)\,\big(\theta(x)-\theta(y)\big)^2}{|x-y|^{N+2s}}\,dx\,dy ,
\end{equation}
all three quantities being finite, and
\begin{equation}\label{eq:Rbound}
\big|R_\theta(u)\big|\ \le\ \frac{C_{N,s}}{2}\,c_{N,s}\,L^{2s}\,\|u\|^2_{L^2(\R^N)} .
\end{equation}
\end{lemma}

\begin{proof}
Both $\theta$ and $\theta^2$ take values in $[0,1]$ and are Lipschitz, so
$\theta u,\theta^2u\in\dHs$ by Lemma~\ref{lem:product}. For real numbers
$a,b,p,q$ one has the algebraic identity
\begin{equation}\label{eq:algebraic}
(ap-bq)^2=(p-q)\big(a^2p-b^2q\big)+pq\,(a-b)^2 ,
\end{equation}
as one checks by expanding both sides. The left side is $a^2p^2-2abpq+b^2q^2$ and
the right side is $a^2p^2+b^2q^2-(a^2+b^2)pq+(a^2+b^2)pq-2abpq$.

Before integrating, we check that the two terms on the right of
\eqref{eq:algebraic} produce absolutely convergent integrals. Note that $u$ is not
assumed to have a sign, so positivity is not used in this estimate. For the second term,
$|u(x)u(y)|\le\frac12\big(u(x)^2+u(y)^2\big)$. The kernel
$\big(\theta(x)-\theta(y)\big)^2|x-y|^{-N-2s}$ is symmetric in $(x,y)$, so by
Tonelli's theorem and Lemma~\ref{lem:elem}
\begin{equation}\label{eq:abscv}
\begin{aligned}
\iint_{\R^{2N}}\frac{|u(x)u(y)|\big(\theta(x)-\theta(y)\big)^2}{|x-y|^{N+2s}}\,dx\,dy
&\le\int_{\R^N}u(y)^2\Big[\int_{\R^N}
\frac{\big(\theta(x)-\theta(y)\big)^2}{|x-y|^{N+2s}}\,dx\Big]dy\\
&\le c_{N,s}\,L^{2s}\,\|u\|^2_{L^2(\R^N)},
\end{aligned}
\end{equation}
which is finite, $L$ denoting the Lipschitz constant of $\theta$. For the first term,
the Cauchy-Schwarz inequality on $\R^{2N}$ gives
\[
\iint\frac{|u(x)-u(y)|\,\big|\theta^2u(x)-\theta^2u(y)\big|}{|x-y|^{N+2s}}\,dx\,dy
\ \le\ \frac{2}{C_{N,s}}\,[u]_s\,[\theta^2u]_s<\infty .
\]

Now apply \eqref{eq:algebraic} with $a=\theta(x)$, $b=\theta(y)$, $p=u(x)$,
$q=u(y)$, divide by $|x-y|^{N+2s}$ and integrate over $\R^{2N}$ against
$\tfrac{C_{N,s}}{2}\,dx\,dy$. Both terms on the right are absolutely convergent by
the previous paragraph, so the integral splits. The left-hand side gives
$[\theta u]_s^2$ and the first term on the right gives $\mathcal E(u,\theta^2u)$.
This is \eqref{eq:identity}, and all three quantities are finite. The bound
\eqref{eq:Rbound} is \eqref{eq:abscv} multiplied by $\tfrac{C_{N,s}}{2}$.
\end{proof}

\subsection{The truncated concentrating family}\label{sec:trunc}

The remainder in \eqref{eq:identity} evaluated on the concentrating family is the
error term generated by the cutoff.

\begin{proposition}\label{prop:Eeps}
Let $\eta,\rho$ be as in \eqref{eq:cutoff} and set
\[
E_\varepsilon:=\frac{C_{N,s}}{2}\iint_{\R^{2N}}
\frac{U_{\varepsilon,j}(x)\,U_{\varepsilon,j}(y)\,\big(\eta(x)-\eta(y)\big)^2}
{|x-y|^{N+2s}}\,dx\,dy .
\]
There is $C=C(N,s,\mu_j,\rho,\eta)>0$ such that
\[
0\ \le\ E_\varepsilon\ \le\ C\,\varepsilon^{\,2\beta_j+2s-N} \, \text{for all }\varepsilon\in(0,\rho/4] ,
\]
and $2\beta_j+2s-N>2s$, so that $E_\varepsilon=o(\varepsilon^{2s})$.
\end{proposition}

\begin{proof}
Write $U:=U^{(j)}$, $\beta:=\beta_j$, $\beta_-:=N-2s-\beta$, and let $C_0\ge1$ be the
constant in \eqref{eq:twoside} for $\mu_j$. Substituting $x=a_j+\varepsilon\xi$,
$y=a_j+\varepsilon\zeta$ and using
$U_{\varepsilon,j}(a_j+\varepsilon\xi)=\varepsilon^{-\frac{N-2s}{2}}U(\xi)$, the
factors combine to
$\varepsilon^{-(N-2s)}\varepsilon^{-N-2s}\varepsilon^{2N}=1$, so
\[
E_\varepsilon=\frac{C_{N,s}}{2}\,I_R,\,
I_R:=\iint_{\R^{2N}}\frac{U(\xi)U(\zeta)\big(\chi(\xi)-\chi(\zeta)\big)^2}
{|\xi-\zeta|^{N+2s}}\,d\xi\,d\zeta ,
\]
where $\chi(\xi):=\eta(a_j+\varepsilon\xi)$ and $R:=\rho/\varepsilon\ge4$. Then
$\chi$ takes values in $[0,1]$, $\chi\equiv1$ on $B_{R/2}$, $\chi\equiv0$ outside
$B_R$, and $\chi$ is Lipschitz with constant
$\varepsilon\operatorname{Lip}(\eta)=\Lambda/R$, $\Lambda:=\rho\operatorname{Lip}(\eta)$.
It suffices to prove
\begin{equation}\label{eq:IR}
I_R\le C_1\,R^{\,N-2s-2\beta}
\end{equation}
for $R\ge4$, since $\varepsilon=\rho/R$ turns \eqref{eq:IR} into the assertion. We prove it in the following three steps.

\emph{Step 1.} The integrand vanishes unless $\max\{|\xi|,|\zeta|\}>R/2$, since $\chi\equiv1$ on $B_{R/2}$. Since the integrand is nonnegative and symmetric,
\[
I_R\ \le\ 2\,J_R,\,
J_R:=\int_{|\zeta|>R/2}U(\zeta)\,K(\zeta)\,d\zeta,\,
K(\zeta):=\int_{\R^N}\frac{U(\xi)\big(\chi(\xi)-\chi(\zeta)\big)^2}
{|\xi-\zeta|^{N+2s}}\,d\xi ,
\]
by Tonelli's theorem.

\emph{Step 2.} Fix $\zeta$ with $|\zeta|>R/2\ge2$. Split
$\R^N=A_1\cup A_2$ with $A_1:=\{|\xi-\zeta|\le|\zeta|/2\}$ and
$A_2:=\{|\xi-\zeta|>|\zeta|/2\}$.

On $A_1$ we have $|\xi|\ge|\zeta|/2>1$, so \eqref{eq:twoside} gives
$U(\xi)\le C_0(|\zeta|/2)^{-\beta}$, and Lemma~\ref{lem:elem} applied to $\chi$ yields
\[
\int_{A_1}\le C_02^{\beta}|\zeta|^{-\beta}
\int_{\R^N}\frac{\big(\chi(\xi)-\chi(\zeta)\big)^2}{|\xi-\zeta|^{N+2s}}d\xi
\ \le\ C_02^{\beta}c_{N,s}\Lambda^{2s}\,|\zeta|^{-\beta}R^{-2s} .
\]

On $A_2$ we use $(\chi(\xi)-\chi(\zeta))^2\le1$ and split further. If
$|\xi|\le2|\zeta|$, then $|\xi-\zeta|^{-N-2s}\le2^{N+2s}|\zeta|^{-N-2s}$ and, by
\eqref{eq:twoside} together with $\beta_-<N-2s<N$ and $\beta<N$,
\[
\int_{|\xi|\le2|\zeta|}U(\xi)\,d\xi
\le C_0\omega_{N-1}\Big(\int_0^1r^{N-1-\beta_-}dr+\int_1^{2|\zeta|}r^{N-1-\beta}dr\Big)
\le C_2|\zeta|^{N-\beta},
\]
where we used $|\zeta|\ge1$. This gives us at most
$2^{N+2s}C_2|\zeta|^{-\beta-2s}$. If $|\xi|>2|\zeta|$, then
$|\xi-\zeta|\ge|\xi|-|\zeta|\ge|\xi|/2$, so this part contributes at most
\[
2^{N+2s}C_0\omega_{N-1}\int_{2|\zeta|}^{\infty}r^{-\beta-N-2s}r^{N-1}dr
=C_3|\zeta|^{-\beta-2s} .
\]
Since $|\zeta|>R/2$ we have $|\zeta|^{-2s}\le2^{2s}R^{-2s}$, so both contributions on
$A_2$ are bounded by $C_4|\zeta|^{-\beta}R^{-2s}$. Altogether
\[
K(\zeta)\ \le\ C_5\,|\zeta|^{-\beta}R^{-2s}\, \text{for }|\zeta|>R/2 .
\]

\emph{Step 3.} Using $U(\zeta)\le C_0|\zeta|^{-\beta}$ for
$|\zeta|>R/2>1$ and $2\beta>N$,
\[
J_R\ \le\ C_5C_0R^{-2s}\int_{|\zeta|>R/2}|\zeta|^{-2\beta}d\zeta
=C_5C_0\,\omega_{N-1}R^{-2s}\int_{R/2}^{\infty}r^{N-1-2\beta}dr
=C_6\,R^{-2s}R^{\,N-2\beta},
\]
which is \eqref{eq:IR}. Finally $2\beta+2s-N>2s$ is exactly $2\beta>N$.
\end{proof}

We next estimate the numerator and the denominator of the Rayleigh quotient.

\begin{theorem}\label{thm:trunc}
Let $\eta,\rho$ be as in \eqref{eq:cutoff} and put $v_\varepsilon:=\eta
U_{\varepsilon,j}$. Then $v_\varepsilon\in X\setminus\{0\}$ for every
$\varepsilon>0$, and as $\varepsilon\to0^+$,
\[
[v_\varepsilon]_s^2-\mu_j\int_\Omega\frac{v_\varepsilon^2}{|x-a_j|^{2s}}\,dx
=\big\|U^{(j)}\big\|_{L^{\cstar}}^{\cstar}
+O\big(\varepsilon^{\,2\beta_j+2s-N}\big)+O\big(\varepsilon^{\tau_j}\big), \,
\tau_j=\frac{2N\beta_j}{N-2s}-N ,
\]
and
\[
0\ \le\ \big\|U^{(j)}\big\|_{L^{\cstar}}^{\cstar}
-\int_\Omega|v_\varepsilon|^{\cstar}dx\ \le\ C\varepsilon^{\tau_j} .
\]
Both error exponents exceed $2s$.
\end{theorem}

\begin{proof}
The support of $v_\varepsilon$ lies in $B_\rho(a_j)\subset\Omega$ by
\eqref{eq:cutoff}, and $v_\varepsilon\in\dHs$ by Lemma~\ref{lem:product}. Hence
$v_\varepsilon\in X$, and $v_\varepsilon\ne0$ since $U_{\varepsilon,j}>0$ and
$\eta\equiv1$ on $B_{\rho/2}(a_j)$.

Apply Lemma~\ref{lem:identity} with $\theta=\eta$ and $u=U_{\varepsilon,j}$, then use
$\eta^2U_{\varepsilon,j}\in\dHs$ as a test function in \eqref{eq:weakUeps} of
Lemma~\ref{lem:eqscaled}:
\[
\mathcal E\big(U_{\varepsilon,j},\eta^2U_{\varepsilon,j}\big)
=\mu_j\int_{\R^N}\frac{\eta^2U_{\varepsilon,j}^2}{|x-a_j|^{2s}}\,dx
+\int_{\R^N}\eta^2U_{\varepsilon,j}^{\cstar}\,dx ,
\]
both integrals being finite by Theorem~\ref{thm:hardy} and by
$0\le\eta\le1$. Since $(\eta U_{\varepsilon,j})^2=\eta^2U_{\varepsilon,j}^2$,
\eqref{eq:identity} becomes
\begin{equation}\label{eq:numerator}
[v_\varepsilon]_s^2-\mu_j\int_\Omega\frac{v_\varepsilon^2}{|x-a_j|^{2s}}\,dx
=\int_{\R^N}\eta^2U_{\varepsilon,j}^{\cstar}\,dx+E_\varepsilon ,
\end{equation}
with $E_\varepsilon$ as in Proposition~\ref{prop:Eeps}. For $q\in\{2,\cstar\}$ we have
$0\le1-\eta^q\le1$ and $1-\eta^q=0$ on $B_{\rho/2}(a_j)$, so by
Lemma~\ref{lem:scaling} and Proposition~\ref{prop:tails}(e), for
$\varepsilon\le\rho/2$,
\begin{equation}\label{eq:critcut}
0\le\int_{\R^N}U_{\varepsilon,j}^{\cstar}dx-\int_{\R^N}\eta^qU_{\varepsilon,j}^{\cstar}dx
\le\int_{|x-a_j|>\rho/2}U_{\varepsilon,j}^{\cstar}dx
=\int_{|y|>\rho/(2\varepsilon)}\big(U^{(j)}\big)^{\cstar}dy
\le C\varepsilon^{\tau_j},
\end{equation}
while $\int_{\R^N}U_{\varepsilon,j}^{\cstar}dx=\|U^{(j)}\|^{\cstar}_{L^{\cstar}}$ by
Lemma~\ref{lem:scaling}. Combining \eqref{eq:numerator}, \eqref{eq:critcut} with
$q=2$ and Proposition~\ref{prop:Eeps} gives the first display, \eqref{eq:critcut} with
$q=\cstar$ gives the second, since $|v_\varepsilon|^{\cstar}=\eta^{\cstar}
U_{\varepsilon,j}^{\cstar}$. Both exponents exceed $2s$ by
Proposition~\ref{prop:tails}(e) and Proposition~\ref{prop:Eeps}.
\end{proof}

\subsection{Expansion of the Rayleigh quotient}\label{sec:expansion}

Combining Theorem~\ref{thm:trunc} with Corollary~\ref{cor:Bj} yields the following
expansion.

\begin{theorem}\label{thm:expansion}
Let $\lambda\in\R$, let $j$, $\rho$ and $\eta$ satisfy \eqref{eq:cutoff}, and let
\[
R_\lambda(v):=\frac{Q_{\boldsymbol\mu,\mathbf a}(v)-\lambda\|v\|^2_{L^2(\Omega)}}
{\|v\|^2_{L^{\cstar}(\Omega)}},\, v\in X\setminus\{0\} .
\]
Then, with $v_\varepsilon=\eta U_{\varepsilon,j}$ and as $\varepsilon\to0^+$,
\[
R_\lambda(v_\varepsilon)=S_{\mu_j}
-C_{\mu_j}\,B_j(\lambda,\mathbf a,\boldsymbol\mu)\,\varepsilon^{2s}
+O\Big(\varepsilon^{\,2s+\kappa_j}\log\tfrac1\varepsilon\Big),
\]
where
\[
C_{\mu_j}:=\frac{\big\|U^{(j)}\big\|_{L^2}^2}{\big\|U^{(j)}\big\|_{L^{\cstar}}^{2}}>0, \,
\kappa_j:=\min\{1,\,2\beta_j-N\}>0 ,
\]
and $B_j(\lambda,\mathbf a,\boldsymbol\mu)$ as in Corollary~\ref{cor:Bj}.
\end{theorem}

\begin{proof}
Write $W(x):=\lambda+\sum_{i\neq j}\mu_i|x-a_i|^{-2s}$ which is smooth on
$\overline{B_\rho(a_j)}$. By the proof of Corollary~\ref{cor:Bj} and let
$M:=\|W\|_{L^\infty(B_\rho(a_j))}$. Separating the contribution of the pole $a_j$,
\[
Q_{\boldsymbol\mu,\mathbf a}(v_\varepsilon)-\lambda\|v_\varepsilon\|_{L^2}^2
=\Big([v_\varepsilon]_s^2-\mu_j\int_\Omega\frac{v_\varepsilon^2}{|x-a_j|^{2s}}\Big)
-\int_{B_\rho(a_j)}W\,\eta^2U_{\varepsilon,j}^2\,dx .
\]
For the second term, $1-\eta^2$ vanishes on $B_{\rho/2}(a_j)$, so by
Lemma~\ref{lem:scaling} and Proposition~\ref{prop:tails}(a), for
$\varepsilon\le\rho/2$,
\[
\Big|\int_{B_\rho(a_j)}W(1-\eta^2)U_{\varepsilon,j}^2\,dx\Big|
\le M\int_{|x-a_j|>\rho/2}U_{\varepsilon,j}^2\,dx
=M\varepsilon^{2s}\!\!\int_{|y|>\rho/(2\varepsilon)}\!\!\big(U^{(j)}\big)^2dy
\le C\varepsilon^{\,2s+2\beta_j-N},
\]
so Corollary~\ref{cor:Bj} gives
\[
\int_{B_\rho(a_j)}W\eta^2U_{\varepsilon,j}^2\,dx
=B_j(\lambda,\mathbf a,\boldsymbol\mu)\big\|U^{(j)}\big\|^2_{L^2}\varepsilon^{2s}
+O\Big(\varepsilon^{2s+\kappa_j}\log\tfrac1\varepsilon\Big) .
\]
Together with Theorem~\ref{thm:trunc} and \eqref{eq:SmuUmu}, and writing
$A:=\|U^{(j)}\|^{\cstar}_{L^{\cstar}}=S_{\mu_j}\|U^{(j)}\|^2_{L^{\cstar}}$,
\[
Q_{\boldsymbol\mu,\mathbf a}(v_\varepsilon)-\lambda\|v_\varepsilon\|^2_{L^2}
=A-B_j\big\|U^{(j)}\big\|^2_{L^2}\varepsilon^{2s}
+O\Big(\varepsilon^{2s+\kappa_j}\log\tfrac1\varepsilon\Big)
+O\big(\varepsilon^{\tau_j}\big)+O\big(\varepsilon^{\,2\beta_j+2s-N}\big).
\]
Now $\kappa_j\le2\beta_j-N$ by definition, and
\[
(\tau_j-2s)-(2\beta_j-N)=\frac{2N\beta_j}{N-2s}-2\beta_j-2s
=2s\Big(\frac{2\beta_j}{N-2s}-1\Big)>0
\]
since $2\beta_j>N>N-2s$. Hence $\tau_j-2s>2\beta_j-N\ge\kappa_j$ and
$(2\beta_j+2s-N)-2s=2\beta_j-N\ge\kappa_j$, so all three error terms are
$O(\varepsilon^{2s+\kappa_j}\log\frac1\varepsilon)$. For the denominator,
Theorem~\ref{thm:trunc} gives
\[
\|v_\varepsilon\|^2_{L^{\cstar}(\Omega)}
=\big(A+O(\varepsilon^{\tau_j})\big)^{2/\cstar}
=\big\|U^{(j)}\big\|^2_{L^{\cstar}}\Big(1+O\big(\varepsilon^{\tau_j}\big)\Big).
\]
Dividing and using $\tau_j-2s>\kappa_j$ once more,
\[
R_\lambda(v_\varepsilon)
=\frac{A}{\|U^{(j)}\|^2_{L^{\cstar}}}
-\frac{\|U^{(j)}\|^2_{L^2}}{\|U^{(j)}\|^2_{L^{\cstar}}}B_j\varepsilon^{2s}
+O\Big(\varepsilon^{2s+\kappa_j}\log\tfrac1\varepsilon\Big)
=S_{\mu_j}-C_{\mu_j}B_j\varepsilon^{2s}
+O\Big(\varepsilon^{2s+\kappa_j}\log\tfrac1\varepsilon\Big). \qedhere
\]
\end{proof}

\subsection{The geometric threshold}\label{sec:geo}

The expansion turns the sign of $B_j$ into the strict inequality
$S_{\lambda,\boldsymbol\mu,\mathbf a}(\Omega)<S_*$, and the sign of $B_j$ at a
dominant pole is exactly the condition $\lambda>\lambda_{\mathrm{geo}}$.

\begin{corollary}\label{cor:strict}
Assume \eqref{eq:H}, let $\lambda<\lambda_1(\boldsymbol\mu,\mathbf a)$, and set
\[
S_{\lambda,\boldsymbol\mu,\mathbf a}(\Omega)
:=\inf_{v\in X\setminus\{0\}}R_\lambda(v) .
\]
If there is $j$ with $\mu_j>0$, $\beta_+(\mu_j)>\frac N2$ and
$B_j(\lambda,\mathbf a,\boldsymbol\mu)>0$, then
\[
S_{\lambda,\boldsymbol\mu,\mathbf a}(\Omega)<S_{\mu_j} .
\]
If moreover $j\in J_*$, then $S_{\lambda,\boldsymbol\mu,\mathbf a}(\Omega)<S_*$. In
particular, if $0<\mu_{\max}<\mubar$, then
$S_{\lambda,\boldsymbol\mu,\mathbf a}(\Omega)<S_*$ for every
\[
\lambda_{\mathrm{geo}}<\lambda<\lambda_1(\boldsymbol\mu,\mathbf a) ,
\]
with $\lambda_{\mathrm{geo}}$ as in \eqref{eq:lambdageo},
and $\lambda_{\mathrm{geo}}<0$ as soon as $k\ge2$ and at least two of the masses are
positive, so that the range then contains $\lambda=0$.
\end{corollary}

\begin{proof}
By Theorem~\ref{thm:expansion}, $R_\lambda(v_\varepsilon)<S_{\mu_j}$ for all small
$\varepsilon>0$, since $C_{\mu_j}B_j>0$ and the error is
$o(\varepsilon^{2s})$. If $j\in J_*$ then $S_{\mu_j}=S_*$ by
Proposition~\ref{prop:Smu}. For the last statement, let $j\in J_*$. Then
$\mu_j=\mu_{\max}$, which is positive by hypothesis and smaller than $\mubar$, so
$\beta_+(\mu_j)>\frac N2$ by Theorem~\ref{thm:mubar}. Moreover
$B_j(\lambda,\mathbf a,\boldsymbol\mu)>0$ for some $j\in J_*$ if and only if
$\lambda>\lambda_{\mathrm{geo}}$. Finally $\lambda_{\mathrm{geo}}<0$ exactly when
$\sum_{i\neq j}\mu_i>0$ for some $j\in J_*$, which holds as soon as at least two of
the masses are positive.
\end{proof}

\begin{remark}[the case $k=1$]\label{rem:GRSZsection6}
For $k=1$ the expansion above is closely related to \cite[Section~6]{GRSZ}, where the cutoff
identity, the remainder $E_\varepsilon$ and its order all appear, for a general
H\"older coefficient $a$ producing $a(0)$ in place of $B_j$. In the several-pole
setting the linear term and the Hardy terms centred at the remaining poles combine
near $a_j$ into $W(x)=\lambda+\sum_{i\neq j}\mu_i|x-a_i|^{-2s}$, smooth on
$\overline{B_\rho(a_j)}$ by \eqref{eq:cutoff}, whose value at $a_j$ is
$B_j(\lambda,\mathbf a,\boldsymbol\mu)$. This yields the threshold $\lambda_{\mathrm{geo}}$ and admits $\lambda=0$ and negative values of $\lambda$,
which does not arise when $k=1$ and $B_j=\lambda$.
\end{remark}

\section{Compactness of Palais-Smale sequences}\label{sec:ps}

Throughout this section we assume \eqref{eq:H} and
$\lambda<\lambda_1(\boldsymbol\mu,\mathbf a)$. For concentration-compactness, profile decompositions and global
compactness in fractional Sobolev spaces we refer to Lions \cite{Lions} and to
Palatucci and Pisante \cite{PP14,PP15}. By
Proposition~\ref{prop:reduction} the compactness statement needed here reduces to a single estimate, rather than a full profile decomposition or global compactness theorem. We call $(u_n)\subset X$ a Palais-Smale sequence for $J_\lambda$ at
level $c\in\R$ if $J_\lambda(u_n)\to c$ and $J_\lambda'(u_n)\to0$ in $X'$.

\subsection{Basic properties of Palais-Smale sequences}\label{sec:psbasic}

We first record that $J_\lambda$ is of class $C^1$, that Palais-Smale sequences are
bounded, and that their weak limits are critical points.

\begin{lemma}\label{lem:C1}
$J_\lambda\in C^1(X)$, with
\[
J_\lambda'(v)[\varphi]=\mathcal E(v,\varphi)
-\sum_i\mu_i\int_\Omega\frac{v\varphi}{|x-a_i|^{2s}}
-\lambda\int_\Omega v\varphi-\int_\Omega|v|^{\cstar-2}v\varphi, \, v,\varphi\in X,
\]
and every Palais-Smale sequence for $J_\lambda$ is bounded in $X$.
\end{lemma}

\begin{proof}
By Proposition~\ref{prop:form} the map $v\mapsto Q_{\boldsymbol\mu,\mathbf a}(v)$ is a
bounded quadratic form on $X$, hence smooth, and so is $v\mapsto\|v\|^2_{L^2(\Omega)}$
since $X\hookrightarrow L^2(\Omega)$. The map $v\mapsto\frac1{\cstar}\int|v|^{\cstar}$
is of class $C^1$ on $L^{\cstar}(\Omega)$ with derivative
$\varphi\mapsto\int|v|^{\cstar-2}v\varphi$ \cite[Ch.~1]{Willem}, and
$X\hookrightarrow L^{\cstar}(\Omega)$ continuously, composing gives the formula.

Let $(u_n)$ be a Palais-Smale sequence at level $c$. Since
$\frac12-\frac1{\cstar}=\frac sN$,
\[
J_\lambda(u_n)-\frac{1}{\cstar}J_\lambda'(u_n)[u_n]
=\frac sN\Big(Q_{\boldsymbol\mu,\mathbf a}(u_n)-\lambda\|u_n\|_{L^2(\Omega)}^2\Big)
\ \ge\ \frac sN\,\theta\,[u_n]_s^2
\]
by Theorem~\ref{thm:A}(v), where $\theta=\theta(\lambda)>0$. The left-hand side is at
most $C+\frac{1}{\cstar}\|J'_\lambda(u_n)\|_{X'}[u_n]_s=C+o(1)[u_n]_s$, so
$[u_n]_s^2\le C'(1+[u_n]_s)$ and $([u_n]_s)$ is bounded.
\end{proof}

Weak limits of Palais-Smale sequences are critical points of $J_\lambda$.

\begin{lemma}\label{lem:weaklimit}
Let $(u_n)$ be a Palais-Smale sequence for $J_\lambda$ at level $c$ and let
$u_n\rightharpoonup u$ in $X$ along a subsequence. Then $u$ is a critical point of
$J_\lambda$, and
\[
J_\lambda(u)=\frac sN\int_\Omega|u|^{\cstar}dx\ \ge\ 0 .
\]
\end{lemma}

\begin{proof}
We fix $\varphi\in X$. Then $\mathcal E(u_n,\varphi)\to\mathcal E(u,\varphi)$ since
$\mathcal E=\langle\cdot,\cdot\rangle_X$ on $X$, each map
$v\mapsto\int_\Omega\frac{v\varphi}{|x-a_i|^{2s}}$ is a bounded linear functional on
$X$. By Cauchy-Schwarz and \eqref{eq:hardy-i}, so these terms also pass to the limit,
and so does $v\mapsto\int_\Omega v\varphi$. For the critical term, $(u_n)$ is bounded
in $L^{\cstar}(\Omega)$ and, along a further subsequence, $u_n\to u$ a.e.\ in
$\Omega$. Hence $\big(|u_n|^{\cstar-2}u_n\big)$ is bounded in
$L^{\cstar/(\cstar-1)}(\Omega)$ and converges a.e.\ to $|u|^{\cstar-2}u$, so it
converges to $|u|^{\cstar-2}u$ weakly in $L^{\cstar/(\cstar-1)}(\Omega)$
\cite[Prop.~5.4.7]{Willem}, and therefore
$\int|u_n|^{\cstar-2}u_n\varphi\to\int|u|^{\cstar-2}u\varphi$. Since
$J'_\lambda(u_n)[\varphi]\to0$, we get $J'_\lambda(u)[\varphi]=0$. Taking $\varphi=u$ gives
$Q_{\boldsymbol\mu,\mathbf a}(u)-\lambda\|u\|^2_{L^2}=\|u\|^{\cstar}_{L^{\cstar}}$,
so that
$$J_\lambda(u)=\big(\frac12-\frac1{\cstar}\big)\|u\|^{\cstar}_{L^{\cstar}}
=\frac sN\|u\|^{\cstar}_{L^{\cstar}}\ge0.$$ This finishes the proof.
\end{proof}

\subsection{Reduction to the critical-norm dichotomy}\label{sec:reduction}

These results reduce the compactness argument to one remaining estimate.

\begin{proposition}\label{prop:reduction}
Let $\lambda<\lambda_1(\boldsymbol\mu,\mathbf a)$ and let $(u_n)$ be a Palais-Smale
sequence for $J_\lambda$ at level $c$. Up to a subsequence, $u_n\rightharpoonup u$ in
$X$, $w_n:=u_n-u\to0$ in $L^q(\Omega)$ for every $q\in[1,\cstar)$, and
$\|w_n\|^{\cstar}_{L^{\cstar}(\Omega)}\to\nu$ for some $\nu\ge0$. For any such
subsequence,
\begin{equation}\label{eq:energysplit}
c=J_\lambda(u)+\frac sN\,\nu,\, J_\lambda(u)\ge0 .
\end{equation}
If in addition
\begin{equation}\label{eq:dichotomy}
\nu=0\, \text{ or }\, \nu\ \ge\ S_*^{N/(2s)} ,
\end{equation}
and $c<c_*$, then $\nu=0$ and $u_n\to u$ strongly in $X$.
\end{proposition}

\begin{proof}
Boundedness follows from Lemma~\ref{lem:C1}, so a subsequence converges weakly in $X$ and, by
the compact embedding $X\hookrightarrow L^q(\Omega)$ for $q<\cstar$, strongly in
$L^q(\Omega)$ and, along a further subsequence, a.e.\ in $\Omega$. We now apply the Brezis-Lieb lemma. Using the boundedness of
$(\|w_n\|_{L^{\cstar}})$ up to a subsequence, $\|w_n\|^{\cstar}_{L^{\cstar}}$ converges.

From $J'_\lambda(u_n)[u_n]=o(1)[u_n]_s=o(1)$,
\[
Q_{\boldsymbol\mu,\mathbf a}(u_n)-\lambda\|u_n\|^2_{L^2}
=\|u_n\|^{\cstar}_{L^{\cstar}}+o(1), \, \text{so}\,
J_\lambda(u_n)=\frac sN\|u_n\|^{\cstar}_{L^{\cstar}}+o(1),
\]
and letting $n\to\infty$, $c=\frac sN\lim_n\|u_n\|^{\cstar}_{L^{\cstar}}$. By the
Brezis-Lieb lemma \cite{BL}, $\|u_n\|^{\cstar}_{L^{\cstar}}
=\|u\|^{\cstar}_{L^{\cstar}}+\|w_n\|^{\cstar}_{L^{\cstar}}+o(1)$, so
$c=\frac sN\|u\|^{\cstar}_{L^{\cstar}}+\frac sN\nu$, which is
\eqref{eq:energysplit} by Lemma~\ref{lem:weaklimit}.

We assume \eqref{eq:dichotomy} and $c<c_*$. If $\nu>0$ then $\nu\ge S_*^{N/(2s)}$,
and \eqref{eq:energysplit} gives $c\ge\frac sN S_*^{N/(2s)}=c_*$, a contradiction.
Hence $\nu=0$, i.e.\ $u_n\to u$ in $L^{\cstar}(\Omega)$. Then
\[
Q_{\boldsymbol\mu,\mathbf a}(u_n)-\lambda\|u_n\|^2_{L^2}
=\|u_n\|^{\cstar}_{L^{\cstar}}+o(1)\longrightarrow\|u\|^{\cstar}_{L^{\cstar}}
=Q_{\boldsymbol\mu,\mathbf a}(u)-\lambda\|u\|^2_{L^2},
\]
the last equality by Lemma~\ref{lem:weaklimit}. Since $\|u_n\|_{L^2}\to\|u\|_{L^2}$,
this gives
$\|u_n\|^2_{\boldsymbol\mu,\mathbf a}\to\|u\|^2_{\boldsymbol\mu,\mathbf a}$, where
$\|\cdot\|_{\boldsymbol\mu,\mathbf a}$ is the norm of Proposition~\ref{prop:form}.
This gives us
$u_n\rightharpoonup u$ with respect to $\mathcal B_{\boldsymbol\mu,\mathbf a}$, since
$\|\cdot\|_{\boldsymbol\mu,\mathbf a}$ and $[\,\cdot\,]_s$ induce the same weak
topology. Thus, $\mathcal B_{\boldsymbol\mu,\mathbf a}(u_n,u)\to
\|u\|^2_{\boldsymbol\mu,\mathbf a}$ and
\[
\|u_n-u\|^2_{\boldsymbol\mu,\mathbf a}
=\|u_n\|^2_{\boldsymbol\mu,\mathbf a}
-2\,\mathcal B_{\boldsymbol\mu,\mathbf a}(u_n,u)
+\|u\|^2_{\boldsymbol\mu,\mathbf a}\longrightarrow0 ,
\]
so $[u_n-u]_s\to0$ by \eqref{eq:equiv}.
\end{proof}

Thus it remains only to prove \eqref{eq:dichotomy}.

\subsection{Localization of the Gagliardo seminorm}\label{sec:localization}

The remaining ingredient is Theorem~\ref{thm:star}, the localized form of the
statement that a weakly vanishing sequence satisfies the lower bound determined by
$S_*$. Its proof uses no profile decomposition, no classification of limit profiles,
and no passage to the Caffarelli-Silvestre extension \cite{CS}. The cutoff identity of
Lemma~\ref{lem:identity} gives an exact localization of the Gagliardo seminorm, while
the additional terms are controlled by \eqref{eq:Rbound} and thereby vanish along a
sequence converging to zero in $L^2$. We begin with the partition of unity adapted to
the poles.

\begin{lemma}[partition of unity]\label{lem:partition}
Let $0<\delta<\frac16\min\big\{d,\ \min_i\operatorname{dist}(a_i,\partial\Omega)\big\}$
and fix $\chi\in C^\infty([0,\infty);[0,1])$ with $\chi\equiv1$ on $[0,1]$ and
$\chi\equiv0$ on $[2,\infty)$. Put $\psi_i(x):=\chi\big(|x-a_i|/\delta\big)$ for
$i=1,\dots,k$ and
\[
\theta_i:=\sin\Big(\frac\pi2\psi_i\Big)\quad(1\le i\le k), \,
\theta_0:=\prod_{i=1}^{k}\cos\Big(\frac\pi2\psi_i\Big).
\]
Then $\theta_0,\dots,\theta_k$ are smooth, $[0,1]$-valued and Lipschitz, and
\begin{enumerate}
\item[(i)] $\sum_{i=0}^{k}\theta_i^2\equiv1$ on $\R^N$,
\item[(ii)] for $1\le i\le k$, $\theta_i\equiv1$ on $B_\delta(a_i)$ and
$\theta_i\equiv0$ outside $B_{2\delta}(a_i)$. The balls $B_{2\delta}(a_i)$ are
pairwise disjoint and contained in $\Omega$,
\item[(iii)] $\theta_0\equiv0$ on $\bigcup_iB_\delta(a_i)$ and $\theta_0\equiv1$
outside $\bigcup_iB_{2\delta}(a_i)$,
\item[(iv)] $\sum_{i=0}^{k}\theta_i^{\cstar}\equiv1$ on
$\R^N\setminus K_\delta$, where
$K_\delta:=\bigcup_{i=1}^k\big\{\delta\le|x-a_i|\le2\delta\big\}$.
\end{enumerate}
\end{lemma}

\begin{proof}
Since $\chi\equiv1$ near $0$, each $\psi_i$ is constant near $a_i$ and hence smooth,
it is supported in $\overline{B_{2\delta}(a_i)}$, and $6\delta<d$ makes these sets
pairwise disjoint and contained in $\Omega$. All $\theta_i$ are therefore smooth,
take values in $[0,1]$, and have bounded first derivatives, so they are Lipschitz.

For (i), fix $x$. If $x$ lies outside every $\operatorname{supp}\psi_i$, then
$\theta_i(x)=0$ for $i\ge1$ and $\theta_0(x)=1$. Otherwise $x$ lies in exactly one
$\operatorname{supp}\psi_j$. So $\theta_i(x)=0$ for $i\ge1$, $i\neq j$. Also, we have
$\theta_0(x)=\cos\big(\frac\pi2\psi_j(x)\big)$, so that
$\sum_i\theta_i(x)^2=\sin^2+\cos^2=1$. Items (ii) and (iii) are immediate from
$\psi_i\equiv1$ on $B_\delta(a_i)$ and $\psi_i\equiv0$ outside $B_{2\delta}(a_i)$. For
(iv), a point outside $K_\delta$ either lies outside every $\operatorname{supp}\psi_i$,
where $\theta_0=1$ and the others vanish, or lies in some $B_\delta(a_j)$, where
$\theta_j=1$ and the others vanish.
\end{proof}

This partition and the cutoff identity give a localized estimate for the seminorm with an $L^2$ controlled error.

\begin{lemma}[localization estimate]\label{lem:IMS}
Let $\theta_0,\dots,\theta_k:\R^N\to[0,1]$ be Lipschitz, with constants
$L_0,\dots,L_k$, and suppose $\sum_{i=0}^k\theta_i^2\equiv1$. Then, for every
$v\in\dHs\cap L^2(\R^N)$,
\[
\Big|\,[v]_s^2-\sum_{i=0}^{k}\big[\theta_iv\big]_s^2\,\Big|
\ \le\ \frac{C_{N,s}\,c_{N,s}}{2}\Big(\sum_{i=0}^kL_i^{2s}\Big)\,\|v\|^2_{L^2(\R^N)} .
\]
\end{lemma}

\begin{proof}
By Lemma~\ref{lem:identity}, $\theta_iv\in\dHs$ and
$[\theta_iv]_s^2=\mathcal E(v,\theta_i^2v)+R_{\theta_i}(v)$ for each $i$. Summing and
using bilinearity of $\mathcal E$ together with $\sum_i\theta_i^2v=v$,
\[
\sum_{i=0}^{k}[\theta_iv]_s^2
=\mathcal E\Big(v,\sum_{i=0}^k\theta_i^2v\Big)+\sum_{i=0}^kR_{\theta_i}(v)
=[v]_s^2+\sum_{i=0}^kR_{\theta_i}(v),
\]
and \eqref{eq:Rbound} bounds each remainder.
\end{proof}

\subsection{The critical-norm inequality}\label{sec:dicho}

We can now establish the required inequality.

\begin{proof}[Proof of Theorem~\ref{thm:star}]
By the compact embedding $X\hookrightarrow L^2(\Omega)$,  $w_n\to0$ in $L^2(\Omega)$. We
fix a subsequence along which $Q_{\boldsymbol\mu,\mathbf a}(w_n)
-S_*\|w_n\|^2_{L^{\cstar}(\Omega)}$ converges to the left-hand side of
\eqref{eq:star}. Any further subsequence has the same limit, so it suffices to show
that this limit is nonnegative.

Since $w_n\rightharpoonup0$ in $X$ we have $\sup_n[w_n]_s<\infty$, so the sequences
$\big(Q_{\boldsymbol\mu,\mathbf a}(w_n)\big)$ and
$\big(\|w_n\|^{\cstar}_{L^{\cstar}}\big)$ are bounded. Extracting further, we may
therefore assume that $Q_{\boldsymbol\mu,\mathbf a}(w_n)$ converges and that
$\|w_n\|^{\cstar}_{L^{\cstar}}\to\nu$ for some $\nu\ge0$. Moreover
$\overline\Omega$ is a compact metric space, so $C(\overline\Omega)$ is separable and
$\mathcal M(\overline\Omega)=C(\overline\Omega)^{*}$ by the Riesz representation
theorem. Hence, by the Banach-Alaoglu theorem and after one further extraction,
\[
|w_n|^{\cstar}\,\mathcal L^N\!\restriction_\Omega\ \rightharpoonup^*\ \varrho
\quad\text{in }\mathcal M(\overline\Omega),
\, \varrho\ge0,\quad\varrho(\overline\Omega)=\nu<\infty ,
\]
the total mass being $\nu$ since the constant function $1$ lies in
$C(\overline\Omega)$. Fix $\delta$
as in Lemma~\ref{lem:partition} and let $\theta_0,\dots,\theta_k$ and $K_\delta$ be as
there. Note that $w_n\in\dHs\cap L^2(\R^N)$, since $w_n\in X$ vanishes outside
$\Omega$.

\emph{Step 1.} By Lemma~\ref{lem:IMS},
\[
[w_n]_s^2=\sum_{i=0}^{k}\big[\theta_iw_n\big]_s^2+o(1),
\]
since $\|w_n\|_{L^2}\to0$ while the Lipschitz constants of the $\theta_i$ depend
only on $\delta$ and $\chi$.

\emph{Step 2.} Fix $i\in\{1,\dots,k\}$. Since
$\sum_l\theta_l^2\equiv1$,
\[
\int_\Omega\frac{w_n^2}{|x-a_i|^{2s}}\,dx
=\sum_{l=0}^{k}\int_\Omega\frac{(\theta_lw_n)^2}{|x-a_i|^{2s}}\,dx .
\]
For $l\ge1$ with $l\neq i$ the function $\theta_l$ vanishes outside
$B_{2\delta}(a_l)$, where $|x-a_i|\ge d-2\delta>\frac d2$, and $\theta_0$ vanishes on
$B_\delta(a_i)$. In both cases the weight $\theta_l^2|x-a_i|^{-2s}$ is bounded by a
constant depending only on $\delta$ and $d$, so the corresponding integral is at most
$C(\delta,d)\|w_n\|^2_{L^2(\Omega)}\to0$. Hence
\[
\int_\Omega\frac{w_n^2}{|x-a_i|^{2s}}\,dx
=\int_{\R^N}\frac{(\theta_iw_n)^2}{|x-a_i|^{2s}}\,dx+o(1).
\]

\emph{Step 3.} Combining Steps 1 and 2 with \eqref{eq:Q},
\[
Q_{\boldsymbol\mu,\mathbf a}(w_n)
=\big[\theta_0w_n\big]_s^2
+\sum_{i=1}^{k}\Big(\big[\theta_iw_n\big]_s^2
-\mu_i\int_{\R^N}\frac{(\theta_iw_n)^2}{|x-a_i|^{2s}}dx\Big)+o(1) .
\]
Each $\theta_iw_n$ lies in $\dHs$, so the definition \eqref{eq:Smu} of $S_{\mu_i}$,
applied after the translation $x\mapsto x-a_i$, and Theorem~\ref{thm:sobolev} give
\[
\big[\theta_iw_n\big]_s^2-\mu_i\int_{\R^N}\frac{(\theta_iw_n)^2}{|x-a_i|^{2s}}dx
\ \ge\ S_{\mu_i}\big\|\theta_iw_n\big\|^2_{L^{\cstar}},
\,
\big[\theta_0w_n\big]_s^2\ \ge\ S\big\|\theta_0w_n\big\|^2_{L^{\cstar}} .
\]
Since $S_*=\min\{S,S_{\mu_1},\dots,S_{\mu_k}\}$, writing
$A_i^{(n)}:=\int_\Omega\theta_i^{\cstar}|w_n|^{\cstar}dx$ we obtain
\begin{equation}\label{eq:step3}
Q_{\boldsymbol\mu,\mathbf a}(w_n)\ \ge\
S_*\sum_{i=0}^{k}\big(A_i^{(n)}\big)^{2/\cstar}+o(1) .
\end{equation}

\emph{Step 4.} Since $2/\cstar\in(0,1)$, the map
$t\mapsto t^{2/\cstar}$ is subadditive on $[0,\infty)$, so
$\sum_i\big(A_i^{(n)}\big)^{2/\cstar}\ge\big(\sum_iA_i^{(n)}\big)^{2/\cstar}$. By
Lemma~\ref{lem:partition}(iv), $\sum_i\theta_i^{\cstar}=1$ off $K_\delta$ and
$\sum_i\theta_i^{\cstar}\ge0$ everywhere, so
\[
\sum_{i=0}^kA_i^{(n)}\ \ge\ \int_\Omega|w_n|^{\cstar}dx-\int_{K_\delta}|w_n|^{\cstar}dx .
\]

\emph{Step 5.} Let
$\Phi:[0,\infty)\to[0,1]$ be continuous with $\Phi(0)=0$, $\Phi\equiv1$ on $[1,2]$ and
$\Phi\equiv0$ on $[3,\infty)$, and set
$\phi_\delta(x):=\sum_{i=1}^k\Phi\big(|x-a_i|/\delta\big)$, supports are disjoint since $6\delta<d$, so $\phi_\delta\in C(\overline\Omega;[0,1])$
and $\phi_\delta\equiv1$ on $K_\delta$. Therefore
\[
\limsup_n\int_{K_\delta}|w_n|^{\cstar}dx
\ \le\ \lim_n\int_\Omega\phi_\delta|w_n|^{\cstar}dx=\int_{\overline\Omega}\phi_\delta\,d\varrho
\ =:\ \tau(\delta) .
\]
As $\delta\to0^+$ we have $\phi_\delta\to0$ pointwise on $\overline\Omega$: at a point
$x\notin\{a_1,\dots,a_k\}$ one has $\phi_\delta(x)=0$ once $3\delta<\min_i|x-a_i|$,
and $\phi_\delta(a_i)=\Phi(0)=0$. Since $0\le\phi_\delta\le1$ and $\varrho$ is finite,
dominated convergence gives $\tau(\delta)\to0$ as $\delta\to0^+$.

\emph{Conclusion.} Fix $\delta$. Combining Steps 4 and 5 with \eqref{eq:step3} and
using that $t\mapsto t^{2/\cstar}$ is nondecreasing,
\[
\lim_n Q_{\boldsymbol\mu,\mathbf a}(w_n)\ \ge\
S_*\Big(\big(\nu-\tau(\delta)\big)_+\Big)^{2/\cstar},
\]
where $t_+=\max\{t,0\}$. The left-hand side does not depend on $\delta$, so letting
$\delta\to0^+$ gives
$\lim_nQ_{\boldsymbol\mu,\mathbf a}(w_n)\ge S_*\nu^{2/\cstar}
=S_*\lim_n\|w_n\|^2_{L^{\cstar}}$, which is \eqref{eq:star}.
\end{proof}

The required dichotomy now follows from Theorem~\ref{thm:star}.

\begin{corollary}[the dichotomy]\label{cor:dicho}
Let $\lambda<\lambda_1(\boldsymbol\mu,\mathbf a)$, let $(u_n)$ be a Palais-Smale
sequence for $J_\lambda$, and let $u_n\rightharpoonup u$ in $X$ with
$\|w_n\|^{\cstar}_{L^{\cstar}}\to\nu$, $w_n=u_n-u$, as in
Proposition~\ref{prop:reduction}. Then
\[
\nu=0 \,\text{ or }\, \nu\ \ge\ S_*^{N/(2s)} .
\]
\end{corollary}

\begin{proof}
As in the proof of Proposition~\ref{prop:reduction}, $J'_\lambda(u_n)[u_n]=o(1)$ and
$J'_\lambda(u)[u]=0$. Subtracting, and using
$Q_{\boldsymbol\mu,\mathbf a}(u_n)=Q_{\boldsymbol\mu,\mathbf a}(u)
+Q_{\boldsymbol\mu,\mathbf a}(w_n)+2\mathcal B_{\boldsymbol\mu,\mathbf a}(u,w_n)$ with
$\mathcal B_{\boldsymbol\mu,\mathbf a}(u,w_n)\to0$, together with
$\|u_n\|_{L^2}\to\|u\|_{L^2}$. The Brezis-Lieb lemma \cite{BL} gives
$Q_{\boldsymbol\mu,\mathbf a}(w_n)-\|w_n\|^{\cstar}_{L^{\cstar}}\to0$, so
$Q_{\boldsymbol\mu,\mathbf a}(w_n)\to\nu$. Since $w_n\rightharpoonup0$ in $X$,
Theorem~\ref{thm:star} yields
$\nu\ge S_*\nu^{2/\cstar}$. If $\nu>0$, then $\nu^{1-2/\cstar}\ge S_*$, and
$1-\frac{2}{\cstar}=\frac{2s}{N}$.
\end{proof}

\subsection{Compactness below the critical level}\label{sec:compact}

The dichotomy combines with Proposition~\ref{prop:reduction} to give relative
compactness below $c_*$.

\begin{corollary}[compactness below the threshold]\label{cor:compact}
Let $\lambda<\lambda_1(\boldsymbol\mu,\mathbf a)$. Every Palais-Smale sequence for
$J_\lambda$ at a level $c<c_*$ has a subsequence converging strongly in $X$.
\end{corollary}

\begin{proof}
Corollary~\ref{cor:dicho} gives condition \eqref{eq:dichotomy} of Proposition~\ref{prop:reduction}.
\end{proof}

\section{Attainment and proof of the main theorem}\label{sec:existence}

\subsection{Attainment below $S_*$}\label{sec:attain}

The critical-norm inequality also yields attainment by direct minimization.

\begin{theorem}\label{thm:existence}
Assume \eqref{eq:H} and $\lambda<\lambda_1(\boldsymbol\mu,\mathbf a)$, and recall
\[
S_{\lambda,\boldsymbol\mu,\mathbf a}(\Omega)
=\inf_{v\in X\setminus\{0\}}R_\lambda(v)\in(0,\infty).
\]
If $S_{\lambda,\boldsymbol\mu,\mathbf a}(\Omega)<S_*$, then the infimum is attained,
and \eqref{eq:P} has a weak solution $u$ with $u>0$ a.e.\ in $\Omega$ and
\[
J_\lambda(u)=\frac sN\,S_{\lambda,\boldsymbol\mu,\mathbf a}(\Omega)^{N/(2s)}<c_* .
\]
Moreover, $u$ is of least energy $J_\lambda(u)\le J_\lambda(v)$ for every nontrivial
weak solution $v$ of \eqref{eq:P}.
\end{theorem}

\begin{proof}
Write $S_\lambda:=S_{\lambda,\boldsymbol\mu,\mathbf a}(\Omega)$ so that $S_\lambda>0$
follows from Theorem~\ref{thm:A}(v) and Theorem~\ref{thm:sobolev}. Let $(v_n)\subset X$
with $\|v_n\|_{L^{\cstar}(\Omega)}=1$ and
$Q_{\boldsymbol\mu,\mathbf a}(v_n)-\lambda\|v_n\|^2_{L^2}\to S_\lambda$. By
Theorem~\ref{thm:A}(v) the sequence is bounded in $X$, so up to a subsequence
$v_n\rightharpoonup v$ in $X$, $v_n\to v$ in $L^2(\Omega)$ and a.e. Put $w_n:=v_n-v$,
so that $w_n\rightharpoonup0$ in $X$.

Similarly as in Corollary~\ref{cor:dicho},
$Q_{\boldsymbol\mu,\mathbf a}(v_n)=Q_{\boldsymbol\mu,\mathbf a}(v)
+Q_{\boldsymbol\mu,\mathbf a}(w_n)+o(1)$ and $\|v_n\|^2_{L^2}\to\|v\|^2_{L^2}$, while
the Brezis-Lieb lemma gives
$1=\|v\|^{\cstar}_{L^{\cstar}}+\|w_n\|^{\cstar}_{L^{\cstar}}+o(1)$. Passing to a
further subsequence, let $t:=\|v\|^{\cstar}_{L^{\cstar}}\in[0,1]$ and
$\|w_n\|^{\cstar}_{L^{\cstar}}\to1-t$. The limit $\lim_nQ_{\boldsymbol\mu,\mathbf a}(w_n)$
exists, since $Q_{\boldsymbol\mu,\mathbf a}(v_n)-\lambda\|v_n\|^2_{L^2}\to S_\lambda$
and every other term in the splitting converges. By the definition of $S_\lambda$ applied to $v$
(trivially true if $v=0$) and by Theorem~\ref{thm:star},
\[
S_\lambda=\Big(Q_{\boldsymbol\mu,\mathbf a}(v)-\lambda\|v\|^2_{L^2}\Big)
+\lim_nQ_{\boldsymbol\mu,\mathbf a}(w_n)
\ \ge\ S_\lambda\,t^{2/\cstar}+S_*\,(1-t)^{2/\cstar}.
\]
Since $p:=2/\cstar\in(0,1)$ we have $t^{p}+(1-t)^{p}\ge1$, so that
\[
S_\lambda\ \ge\ S_\lambda\big(t^{p}+(1-t)^{p}\big)+(S_*-S_\lambda)(1-t)^{p}
\ \ge\ S_\lambda+(S_*-S_\lambda)(1-t)^{p} .
\]
As $S_*>S_\lambda$ this implies $t=1$. Hence $\|v\|_{L^{\cstar}}=1$, so $v\neq0$, and
the above identity becomes
$S_\lambda=\big(Q_{\boldsymbol\mu,\mathbf a}(v)-\lambda\|v\|^2_{L^2}\big)
+\lim_nQ_{\boldsymbol\mu,\mathbf a}(w_n)$ with both terms at least $S_\lambda$ and
$0$ respectively. Therefore
$Q_{\boldsymbol\mu,\mathbf a}(v)-\lambda\|v\|^2_{L^2}=S_\lambda$ and $v$ is a
minimizer.

Replacing $v$ by $|v|$ changes neither $\|v\|_{L^{\cstar}}$, nor $\|v\|_{L^2}$, nor
the Hardy terms, and does not increase $[\,\cdot\,]_s$, so $|v|$ is a minimizer as
well and we may assume $v\ge0$. As in the proof of Theorem~\ref{thm:A}(iii), the
quotient $R_\lambda$ is of class $C^1$ on $X\setminus\{0\}$ and its derivative
vanishes at $v$, that is
\[
\mathcal B_{\boldsymbol\mu,\mathbf a}(v,\varphi)-\lambda\int_\Omega v\varphi
=S_\lambda\int_\Omega v^{\cstar-1}\varphi \, \text{for all }\varphi\in X .
\]
Setting $u:=\kappa v$ with $\kappa:=S_\lambda^{1/(\cstar-2)}
=S_\lambda^{(N-2s)/(4s)}>0$ turns this into the weak formulation of \eqref{eq:P}, and
$u\ge0$, $u\neq0$. Using $\kappa^{\cstar}=\kappa^2S_\lambda$,
\[
J_\lambda(u)=\Big(\frac12-\frac1{\cstar}\Big)\kappa^2S_\lambda
=\frac sN\,S_\lambda^{\,1+\frac{2}{\cstar-2}}=\frac sN\,S_\lambda^{N/(2s)}
\ <\ \frac sN S_*^{N/(2s)}=c_* .
\]

Finally we show $u>0$ a.e.\ in $\Omega$, for every $\lambda<\lambda_1$. Put
$c:=\min\{\lambda,0\}\le0$, a constant so $c\in L^1_{\mathrm{loc}}(\Omega)$. For
$\varphi\in X$ with $\varphi\ge0$, the weak formulation and $\mu_i\ge0$, $u\ge0$,
$\varphi\ge0$ give
\[
\mathcal E(u,\varphi)
=\sum_i\mu_i\int_\Omega\frac{u\varphi}{|x-a_i|^{2s}}
+\lambda\int_\Omega u\varphi+\int_\Omega u^{\cstar-1}\varphi
\ \ge\ \lambda\int_\Omega u\varphi\ \ge\ \int_\Omega c\,u\varphi ,
\]
the last step since $\lambda\ge c$ and $u\varphi\ge0$. Since $u\in\dHs$ vanishes
outside $\Omega$ and $u\ge0$ a.e.\ in $\R^N$, Theorem~\ref{thm:smp} applies with
$D=\Omega$, either $u>0$ a.e.\ in $\Omega$ or $u=0$ a.e.\ in $\R^N$. Since $u\neq0$, we only have the first possibility.

Let $v$ be any nontrivial weak solution of \eqref{eq:P}. Testing the weak
formulation with $v$ gives
$Q_{\boldsymbol\mu,\mathbf a}(v)-\lambda\|v\|^2_{L^2}=\|v\|^{\cstar}_{L^{\cstar}}$,
so $R_\lambda(v)=\|v\|^{\cstar-2}_{L^{\cstar}}$. Substituting this into \eqref{eq:J}
and using $\frac12-\frac1{\cstar}=\frac sN$ we have,
\[
J_\lambda(v)=\frac sN\|v\|^{\cstar}_{L^{\cstar}}
=\frac sN\,R_\lambda(v)^{\cstar/(\cstar-2)}
=\frac sN\,R_\lambda(v)^{N/(2s)}\ \ge\ \frac sN\,S_\lambda^{N/(2s)}=J_\lambda(u),
\]
since $R_\lambda(v)\ge S_\lambda$ and $\frac{\cstar}{\cstar-2}=\frac{N}{2s}$.
\end{proof}

\subsection[Proof of the main theorem]{Proof of Theorem~\protect\ref{thm:main}}\label{sec:proofmain}

\begin{proof}[Proof of Theorem~\ref{thm:main}]
Under the stated hypotheses, Corollary~\ref{cor:strict} gives
$S_{\lambda,\boldsymbol\mu,\mathbf a}(\Omega)<S_*$ for every
$\lambda_{\mathrm{geo}}<\lambda<\lambda_1(\boldsymbol\mu,\mathbf a)$. The result then
follows from Theorem~\ref{thm:existence}. When at least two masses are positive,
Corollary~\ref{cor:strict} also gives $\lambda_{\mathrm{geo}}<0$, so the existence
range includes $\lambda=0$ and a nonempty interval of negative values.
\end{proof}

\subsection{The upper endpoint and the whole-space problem}\label{sec:endpoint}

\begin{remark}
The upper endpoint in Theorem~\ref{thm:main} is sharp for positive solutions. Also, \eqref{eq:P} has no
solution positive a.e.\ in $\Omega$ when $\lambda\ge\lambda_1(\boldsymbol\mu,\mathbf a)$.
Indeed, let $u>0$ solve \eqref{eq:P} and let $\varphi_1\in X$ be a positive
eigenfunction, which exists by Theorem~\ref{thm:A}. Testing the equation for $u$ with
$\varphi_1$, and the equation for $\varphi_1$ with $u$, and using the symmetry of
$\mathcal B_{\boldsymbol\mu,\mathbf a}$,
\[
\lambda_1\int_\Omega u\varphi_1
=\mathcal B_{\boldsymbol\mu,\mathbf a}(u,\varphi_1)
=\lambda\int_\Omega u\varphi_1+\int_\Omega u^{\cstar-1}\varphi_1 ,
\]
all integrals being finite by Proposition~\ref{prop:form} and H\"older's inequality.
Hence $(\lambda_1-\lambda)\int_\Omega u\varphi_1=\int_\Omega u^{\cstar-1}\varphi_1>0$,
and $\int_\Omega u\varphi_1>0$, so $\lambda<\lambda_1$.
\end{remark}
\begin{remark} The boundedness of $\Omega$ is used in Theorem~\ref{thm:star} through the compact embedding $X\hookrightarrow L^2(\Omega)$. The singularities are treated using the partition of unity introduced in Lemma~\ref{lem:partition}, which relies on the fact that the poles are finitely many and have positive mutual distance. On $\R^N$, however, the constant $S_*$ is no longer the appropriate threshold. Indeed, by dilating a fixed test function so that the relative distances between the poles vanish at the rescaled level, one obtains $$ \inf_{u\in\dHs\setminus\{0\}} \frac{Q_{\boldsymbol\mu,\mathbf a}(u)} {\|u\|^2_{L^{\cstar}(\R^N)}} \leq S_{\sum_i\mu_i} \leq S_*. $$ The second inequality follows from the monotonicity of $\mu\mapsto S_\mu$ and is strict whenever at least two coefficients $\mu_i$ are positive. \end{remark}

\end{document}